\documentclass[12pt,reqno]{amsart}
\usepackage[a4paper,margin=2.5cm,top=2.5cm,bottom=2.5cm,centering,headheight=3ex,headsep=4ex,vcentering]{geometry}

\usepackage{amssymb,amsfonts,amsmath,amsthm}
\usepackage{xcolor}
\usepackage{hyperref}
\usepackage{orcidlink}

\definecolor{equationblue}{HTML}{244A73}
\definecolor{citationgreen}{HTML}{276347}
\definecolor{urlteal}{HTML}{176B75}

\hypersetup{
    colorlinks=true,
    citecolor=green,
    linkcolor=blue,
    urlcolor=blue
}
\usepackage{breqn}
\usepackage{microtype}
\usepackage{mlmodern}

\theoremstyle{plain}

\numberwithin{equation}{section}

\newtheorem{theorem}{Theorem}[section]
\newtheorem{lemma}[theorem]{Lemma}

\newtheorem{proposition}[theorem]{Proposition}

\theoremstyle{definition}

\usepackage{fancyhdr}

\newcommand{\ignore}[1]{}
\newcommand{\assertionproof}[1]{%
  \par\medskip\noindent
  \textbf{Proof of (#1).}\enspace
  \ignorespaces}
\newlength{\auxindent}

\newcommand{\auxidentity}[1]{%
  \par\begingroup
  \raggedright
  \baselineskip=1.25\baselineskip
  \binoppenalty=0
  \relpenalty=10000
  \medmuskip=5mu
  \thickmuskip=6mu
  \let\auxsavedcdot\cdot
  \renewcommand{\cdot}{\nobreak\auxsavedcdot\nobreak}%
  \auxidentitysplit#1\auxidentityend
  \par\endgroup
  \vspace{8pt}
}

\def\auxidentitysplit#1=#2\auxidentityend{%
  \settowidth{\auxindent}{\(\displaystyle #1={}\)}%
  \hangindent=\auxindent
  \hangafter=1
  \noindent\(\displaystyle #1={}#2\)%
}
\title[Partitions with colored odd parts]{Congruences modulo powers of $7$ for partitions with colored odd parts}

\author[S. Maity]{Saikat Maity \orcidlink{0009-0004-7188-471X}}
\address[S. Maity]{Mathematical and Physical Sciences division, School of Arts and Sciences, Ahmedabad University, Ahmedabad 380009, Gujarat, India}
\email{saikat.maity@ahduni.edu.in}

\thanks{Corresponding Author: Saikat Maity (saikat.maity@ahduni.edu.in)}

\author[M. P. Saikia]{Manjil P. Saikia \orcidlink{0000-0002-2997-6731}}
\address[M. P. Saikia]{Mathematical and Physical Sciences division, School of Arts \& Sciences, Ahmedabad University, Navrangpura, Ahmedabad 380009, Gujarat, India}
\email{manjil.saikia@ahduni.edu.in}

\author[P. Talukdar]{Prabal Talukdar \orcidlink{0009-0006-3713-1348}}
\address[P. Talukdar]{Mathematical and Physical Sciences Division, School of Arts \& Sciences, Ahmedabad University, Navrangpura, Ahmedabad 380009, Gujarat, India}
\email{prabaltalukdar89@gmail.com}



\begin{document}
\raggedbottom
\allowdisplaybreaks[1]

\begin{abstract}
Let $a_k(n)$ denote the number of partitions of $n$ in which even parts occur in one color and odd parts may occur in $k$ colors. We prove infinite families of congruences modulo powers of $7$ for $a_k(n)$ when $k\in\{4,5,7\}$. For each positive integer $\alpha$, let $\lambda_{k,\alpha}$ be the least nonnegative solution of $24\lambda_{k,\alpha}\equiv-(k-2)\pmod{7^\alpha}$. Then, we show that
\[
a_k(7^\alpha n+\lambda_{k,\alpha})
\equiv0\pmod{7^{\lfloor(\alpha+1)/2\rfloor}}
\]
for every nonnegative integer $n$. These results extend congruences modulo $7$ of Hirschhorn and Sellers (2026) and complement Chen’s (2025) congruences for the reciprocal crank parity function $a_3(n)$. %The proof uses a degree-seven modular equation, eta-quotient identities, and an alternating $U_7$-operator construction that controls the $7$-adic divisibility of successive coefficient extractions.
\end{abstract}

\maketitle

\section{Introduction}\label{sec:intro}

A partition of a positive integer $n$ is a finite non-increasing sequence of positive integers whose sum is $n$. Let $p(n)$ denote the number of partitions of $n$, with the convention that $p(0)=1$. For example, $p(4)=5$, since the partitions of $4$ are
\[
4,\qquad 3+1,\qquad 2+2,\qquad 2+1+1,\qquad 1+1+1+1.
\]
Since they first appeared, partitions have captured the imagination of mathematicians, new and old. For further background on the theory of partitions, interested readers may refer to Andrews~\cite{Andrews1998}.

Throughout this paper, $q$ is a formal variable; when convergence is relevant, we assume that $q\in\mathbb{C}$ and $|q|<1$. We use the standard notation
\[
(a;q)_\infty := \prod_{j=0}^{\infty}(1-aq^j),
\qquad
f_k := (q^k;q^k)_\infty = \prod_{j=1}^{\infty}(1-q^{kj})
\quad (k\geq 1).
\]
The partition function $p(n)$ has the generating function
\[
\sum_{n=0}^{\infty}p(n)q^n
=
\prod_{j=1}^{\infty}\frac{1}{1-q^j}
=
\frac{1}{f_1}.
\]
Among the most celebrated results in all of mathematics, we have Ramanujan's~\cite{Ramanujan1919} three classical congruences for the partition function: for every integer $n\geq 0$,
\[
\begin{aligned}
p(5n+4)  &\equiv 0 \pmod{5},\\
p(7n+5)  &\equiv 0 \pmod{7},\\
p(11n+6) &\equiv 0 \pmod{11}.
\end{aligned}
\]

To give a combinatorial explanation of these congruences, Dyson~\cite{Dyson1944} defined the \textit{rank} of a partition as its largest part minus its number of parts. He conjectured that grouping partitions by their ranks modulo $5$ and $7$ gives five equal classes for partitions of $5n+4$ and seven equal classes for partitions of $7n+5$, respectively. The rank does not give a similar explanation of the congruence modulo $11$. Dyson therefore suggested that another quantity, which he named the \textit{crank}, might explain this congruence. After several decades of search, Andrews and Garvan~\cite{AndrewsGarvan1988} gave an explicit definition of the crank in 1988.

For a partition $\pi$, let $\ell(\pi)$ be its largest part, $\omega(\pi)$ the number of ones, and $\mu(\pi)$ the number of parts greater than $\omega(\pi)$. The crank of $\pi$ is defined by
\[
\operatorname{crank}(\pi):=
\begin{cases}
\ell(\pi), & \text{if } \omega(\pi)=0,\\
\mu(\pi)-\omega(\pi), & \text{if } \omega(\pi)>0.
\end{cases}
\]
Let $M_e(n)$ and $M_o(n)$ denote the numbers of partitions of $n$ with even and odd crank, respectively. For $n\geq 1$, define the crank parity function by $M(n):=M_e(n)-M_o(n)$, and set $M(0)=1$. Taking $z=-1$ in Garvan's generating function~\cite[Eq.~(3.8)]{Garvan1988} and using \cite[Theorem~8]{Garvan1988}, with the case $n=1$ treated separately, we can prove
\begin{equation}\label{eq:gf_crank_parity}
\sum_{n=0}^{\infty}M(n)q^n
=
2q+\frac{(q;q)_\infty}{(-q;q)_\infty^2}
=
2q+\frac{f_1^3}{f_2^2}.
\end{equation}

Taking the reciprocal of the infinite product in \eqref{eq:gf_crank_parity} gives the generating function for the sequence $a(n)$ studied by Amdeberhan and Merca~\cite{AmdeberhanMerca2025}:
\[
\sum_{n=0}^{\infty}a(n)q^n
=
\frac{(-q;q)_\infty^2}{(q;q)_\infty}
=
\frac{f_2^2}{f_1^3}.
\]
Among the partition interpretations given by Amdeberhan and Merca, $a(n)$ counts partitions of $n$ in which even parts appear in one color (i.e., are monochromatic), while odd parts may appear in three colors. Hirschhorn and Sellers~\cite{HS25} considered the corresponding generalization with any fixed number of colors for the odd parts.

For a fixed integer $k\geq 1$, let $a_k(n)$ denote the number of partitions of $n$ in which even parts appear in one color, while each odd part may be assigned one of $k$ colors. We set $a_k(0)=1$. For example, $a_4(2)=11$, with the eleven colored partitions being
\[
\begin{gathered}
2,\qquad 1_1+1_1,\qquad 1_2+1_2,\qquad 1_3+1_3,\qquad 1_4+1_4,\\[6pt]
1_1+1_2,\qquad 1_1+1_3,\qquad 1_1+1_4,\qquad 1_2+1_3,\qquad 1_2+1_4,\qquad 1_3+1_4.
\end{gathered}
\]
Here the subscripts indicate colors. The generating function for $a_k(n)$ is
\begin{equation}\label{eq:gf_ak}
  \sum_{n=0}^{\infty} a_k(n) q^n = \frac{(q^2;q^2)_\infty^{k-1}}{(q;q)_\infty^k} = \frac{f_2^{ k-1}}{f_1^{ k}}.
\end{equation}
For $k=1,2,3$, this family gives
\[
a_1(n)=p(n),\qquad
a_2(n)=\overline{p}(n),\qquad
a_3(n)=a(n),
\]
where $\overline{p}(n)$ denotes the number of overpartitions of $n$, namely partitions in which the first occurrence of each distinct part may be overlined~\cite{CL04}. Thus, $a_k(n)$ places the ordinary partition, overpartition, and reciprocal crank parity functions in a common setting.

Hirschhorn and Sellers \cite[Theorem 1.3]{HS25} established congruences modulo $7$ for $a_4(7n+4), a_5(7n+6)$, and $a_7(7n+3)$. Chen~\cite[Theorem~1.5]{Chen25} subsequently proved an infinite family of congruences modulo powers of $7$ for $a_3(n)$. In this paper, we adapt the modular-function approach to $a_4(n), a_5(n)$, and $a_7(n)$, extending the three congruences of Hirschhorn and Sellers to infinite families modulo powers of $7$. Our main result is the following.
\begin{theorem}\label{thm:main}
Let $k\in\{4,5,7\}$. For each integer $\alpha\geq 1$, let $\lambda_{k,\alpha}$ be the least nonnegative integer satisfying
\[
24\lambda_{k,\alpha}\equiv -(k-2)\pmod{7^\alpha}.
\]
Then, for every integer $\alpha\geq 1$ and $n\geq 0$,
\[
a_k(7^\alpha n+\lambda_{k,\alpha})
\equiv 0\pmod{7^{\lfloor(\alpha+1)/2\rfloor}}.
\]
\end{theorem}

Further arithmetic properties of this family and its generalizations have been investigated recently by Johnson, Sellers, and Fathima \cite{JSF}, who established congruences modulo powers of $2$ and modulo $11$; Thejitha, Sellers, and Fathima \cite{TSF26}, who studied a two-parameter extension allowing separate numbers of colors for even and odd parts; and Jameson and Sellers \cite{JS26}, who developed a uniform framework for congruences modulo $3$.

The remainder of this paper is organized as follows. Section~\ref{sec:preliminaries} gives the modular equation and the properties of the $U_p$ operator needed in the proof. In Section~\ref{sec:auxiliary_functions}, we introduce the auxiliary functions and state the divisibility result from which Theorem~\ref{thm:main} follows. We prove this result by induction in Section~\ref{sec:proof_main} and then deduce the theorem. Appendices~\ref{app:a4}--\ref{app:a7} contain the modular identities for $a_4(n)$, $a_5(n)$, and $a_7(n)$, respectively, preceded by an account of their discovery and verification.

\section{Preliminaries}\label{sec:preliminaries}
We begin with some definitions and results that will be used in the proof of Theorem~\ref{thm:main}.

Let $\mathbb{H}$ denote the upper half-plane and put $q=e^{2\pi i\tau}$ for $\tau\in\mathbb{H}$. The Dedekind eta function is defined by
\[
\eta(\tau):=q^{1/24}\prod_{n=1}^{\infty}(1-q^n)
=q^{1/24}f_1.
\]
For a positive integer $N$, an eta quotient is a function of the form
\[
\prod_{\delta\mid N}\eta(\delta\tau)^{r_\delta},
\qquad r_\delta\in\mathbb{Z}.
\]

We define the eta quotient
\begin{equation}\label{eq:def_t}
t=t(\tau):=\frac{\eta(7\tau)^4}{\eta(\tau)^4}
=q\frac{f_7^4}{f_1^4},
\end{equation}
which is a Hauptmodul for $\Gamma_0(7)$~\cite[Table~2]{Maier2009}. The relation between $t(\tau)$ and $t(7\tau)$ is given by the following modular equation~\cite[Theorem~2.6]{CCG24}.

\begin{theorem}\label{thm:modular_equation}
Define
\[
\begin{aligned}
a_0(t)&=t,\\
a_1(t)&=7^2t^2+4\cdot7t,\\
a_2(t)&=7^4t^3+4\cdot7^3t^2+46\cdot7t,\\
a_3(t)&=7^6t^4+4\cdot7^5t^3+46\cdot7^3t^2+272\cdot7t,\\
a_4(t)&=7^8t^5+4\cdot7^7t^4+46\cdot7^5t^3+272\cdot7^3t^2+845\cdot7t,\\
a_5(t)&=7^{10}t^6+4\cdot7^9t^5+46\cdot7^7t^4+272\cdot7^5t^3+845\cdot7^3t^2+176\cdot7^2t,\\
a_6(t)&=7^{12}t^7+4\cdot7^{11}t^6+46\cdot7^9t^5+272\cdot7^7t^4+845\cdot7^5t^3+176\cdot7^4t^2+82\cdot7^2t,
\end{aligned}
\]
where $t=t(\tau)$ is defined in \eqref{eq:def_t}. Then
\[
t(\tau)^7-\sum_{l=0}^{6}a_l\bigl(t(7\tau)\bigr)t(\tau)^l=0.
\]
\end{theorem}

To study the coefficients in Theorem~\ref{thm:main}, we use the $U_p$ operator. For a prime $p$ and a formal Laurent series
\[
f=\sum_{m=m_0}^{\infty}a(m)q^m,
\qquad m_0\in\mathbb{Z},
\]
define
\begin{equation}\label{eq:Up_definition}
U_p(f):=\sum_{pm\geq m_0}a(pm)q^m.
\end{equation}
If the series converges, this definition gives
\[
U_p(f)(\tau)
=
\frac{1}{p}\sum_{j=0}^{p-1}
f\left(\frac{\tau+j}{p}\right).
\]
The operator is linear and satisfies
\begin{equation}\label{eq:Up_product}
U_p\bigl(f(\tau)h(p\tau)\bigr)
=
h(\tau)U_p(f)(\tau),
\end{equation}
whenever $f$ and $h$ have formal Laurent expansions with only finitely many negative powers of $q$. Both properties follow directly from \eqref{eq:Up_definition}. 

The following result describes how $U_p$ acts on the level of a modular function.

\begin{theorem}[{\cite[Lemma~7, p.~138]{AtkinLehner1970}}]
\label{thm:Up_level}
Let $p$ be a prime and let $N$ be a positive integer with $p\mid N$. If $f$ is a modular function on $\Gamma_0(pN)$, then $U_p(f)$ is a modular function on $\Gamma_0(N)$.
\end{theorem}
\noindent We now combine the modular equation with the product identity \eqref{eq:Up_product} to obtain a recurrence for $U_7(ut^j)$, called \textit{the fundamental lemma}.
\begin{lemma}[Fundamental lemma]
\label{lem:U7_recurrence}
Let $u=u(\tau)$ have a formal Laurent expansion with only finitely many negative powers of $q$. For every integer $j$,
\[
U_7(ut^j)
=
\sum_{l=0}^{6}a_l(t)U_7(ut^{j+l-7}),
\]
where $t=t(\tau)$ and the polynomials $a_l(t)$ are those given in Theorem~\ref{thm:modular_equation}.
\end{lemma}
\noindent For a proof of this, see \cite[Lemma~3.5]{CCG24}.

\smallskip

To state the next result, we define the order of a nonzero Laurent polynomial $g(t)=\sum_n c_n t^n$ by
\[
\operatorname{ord}_t(g):=\min\{n:c_n\neq 0\},
\]
with the convention that $\operatorname{ord}_t(0)=+\infty$. The next lemma shows that Laurent polynomial representations with prescribed order bounds need only be established for seven consecutive exponents.
\begin{lemma}[{\cite[Lemma~3.6]{CCG24}}]
\label{lem:order_extension}
Let $u,v_1,v_2,v_3:\mathbb{H}\to\mathbb{C}$ and let $j_0\in\mathbb{Z}$. Fix integers $s_1,s_2,s_3$. Suppose that, for each integer $j$ with $j_0\leq j\leq j_0+6$, there are Laurent polynomials
\[
P_j^{(i)}(t)\in\mathbb{Z}[t,t^{-1}],
\qquad i=1,2,3,
\]
satisfying
\[
U_7(ut^j)
=
v_1P_j^{(1)}(t)+v_2P_j^{(2)}(t)+v_3P_j^{(3)}(t)
\]
and
\[
\operatorname{ord}_t\bigl(P_j^{(i)}(t)\bigr)
\geq
\left\lceil\frac{j+s_i}{7}\right\rceil,
\qquad i=1,2,3.
\]
Then such Laurent polynomials exist for every integer $j$, with the same identity and order bounds.
\end{lemma}

We also need to control the powers of $7$ dividing the coefficients. The following lemma extends these coefficient bounds from seven consecutive exponents to all integer exponents.

\begin{lemma}[{\cite[Lemma~3.7]{CCG24}}]
\label{lem:coefficient_extension}
Let $u,v_1,v_2,v_3:\mathbb{H}\to\mathbb{C}$ and let $j_0\in\mathbb{Z}$. Fix integers $r_1,r_2,r_3$. Suppose that, for each integer $j$ with $j_0\leq j\leq j_0+6$, there are Laurent polynomials
\[
P_j^{(i)}(t)\in\mathbb{Z}[t,t^{-1}],
\qquad i=1,2,3,
\]
such that
\[
U_7(ut^j)
=
v_1P_j^{(1)}(t)+v_2P_j^{(2)}(t)+v_3P_j^{(3)}(t),
\]
where
\[
P_j^{(i)}(t)
=
\sum_n c_i(j,n) 
7^{\left\lfloor(7n-j+r_i)/4\right\rfloor}t^n,
\qquad c_i(j,n)\in\mathbb{Z}.
\]
Then these families of Laurent polynomials can be extended to every integer $j$, preserving both the identity and the stated form of their coefficients.
\end{lemma}

\section{Auxiliary functions}\label{sec:auxiliary_functions}

The recursions used below are built from $U_7$ and two auxiliary functions associated with each generating function. To define them, we write
\[
F_k(q):=\sum_{n=0}^{\infty}a_k(n)q^n
=\frac{f_2^{k-1}}{f_1^k},
\qquad k\in\{4,5,7\},
\]
as in \eqref{eq:gf_ak}, and for these values of $k$, define
\[
A_k:=\frac{F_k(q)}{q^{2(k-2)}F_k(q^{49})}
=\frac{f_2^{k-1}f_{49}^{k}}
{q^{2(k-2)}f_1^{k}f_{98}^{k-1}},
\]
and
\[
B_k:=
\begin{cases}
1, & k\in\{4,5\},\\[4pt]
\dfrac{1}{q}, & k=7.
\end{cases}
\]
For a map $h:\mathbb{H}\to\mathbb{C}$ admitting a Laurent expansion in $q$ with only finitely many negative powers, define
\begin{equation}\label{eq:auxiliary_operators}
U_{A_k}(h):=U_7(A_kh),
\qquad
U_{B_k}(h):=U_7(B_kh).
\end{equation}

We also define two auxiliary functions common to all these values of $k$:
\begin{equation}\label{eq:def_p0_p1}
p_0:=q\frac{f_{14}^4f_1^4}{f_7^4f_2^4},
\qquad
p_1:=\frac{1}{7}
\left(\frac{f_{14}f_1^7}{f_7f_2^7}-1\right).
\end{equation}
Starting with $L_{k,0}:=1$, we define a sequence of functions by alternating the two operators. For $k=4,5$ and every integer $\alpha\geq 0$, set
\[
L_{k,2\alpha+1}:=U_{A_k}(L_{k,2\alpha}),
\qquad
L_{k,2\alpha+2}:=U_{B_k}(L_{k,2\alpha+1}).
\]
For $k=7$, we include a factor of $q$ at each odd step:
\[
L_{7,2\alpha+1}:=q\,U_{A_7}(L_{7,2\alpha}),
\qquad
L_{7,2\alpha+2}:=U_{B_7}(L_{7,2\alpha+1}),
\qquad \alpha\geq 0.
\]

Applying \eqref{eq:Up_definition} and \eqref{eq:Up_product} successively to the generating function \eqref{eq:gf_ak}, we obtain, for every integer $\alpha\geq 1$,
\[
\begin{aligned}
L_{k,2\alpha-1}
&=
\frac{1}{F_k(q^7)}
\sum_{n=0}^{\infty}
a_k\bigl(7^{2\alpha-1}n+\lambda_{k,2\alpha-1}\bigr)q^n,\\
L_{k,2\alpha}
&=
\frac{1}{F_k(q)}
\sum_{n=0}^{\infty}
a_k\bigl(7^{2\alpha}n+\lambda_{k,2\alpha}\bigr)q^n,
\end{aligned}
\]
where
\[
\begin{gathered}
\lambda_{k,2\alpha-1}=\lambda_{k,2\alpha}
=\frac{(k-2)(7^{2\alpha}-1)}{24}
\qquad (k=4,5),\\[8pt]
\lambda_{7,2\alpha-1}=\frac{11\cdot7^{2\alpha-1}-5}{24},
\qquad
\lambda_{7,2\alpha}=\frac{5(7^{2\alpha}-1)}{24}.
\end{gathered}
\]
\par\vspace{8pt}

To write the finite expansions arising from these operators in a uniform form, we follow the terminology of Paule and Radu~\cite[Definition~2.8]{PauleRadu2012}. A function $r:\mathbb{Z}\to\mathbb{Z}$ is called \emph{discrete} if $r(n)=0$ for all but finitely many integers $n$. Thus, sums with discrete coefficient functions contain only finitely many nonzero terms. With this convention, we define

\begin{equation}\label{eq:def_XA}\begin{aligned} X_A:=\Bigl\{&7r_1(0)+\sum_{j\ge1}r_1(j)7^{\lfloor\frac{7j+2}{4}\rfloor}t^j+p_0\sum_{j\ge1}r_2(j)7^{\lfloor\frac{7j+2}{4}\rfloor}t^j+p_1\sum_{j\ge0}r_3(j)7^{\lfloor\frac{7j+5}{4}\rfloor}t^j\Bigr\},\end{aligned}\end{equation}

\begin{equation}\label{eq:def_XB}\begin{aligned} X_B:=\Bigl\{&7r_1(-1)t^{-1}+7r_1(0)+\sum_{j\ge1}r_1(j)7^{\lfloor\frac{7j-3}{4}\rfloor}t^j+p_0\Bigl(7r_2(0)+\sum_{j\ge1}r_2(j)7^{\lfloor\frac{7j-3}{4}\rfloor}t^j\Bigr)\\[6pt]&\quad+p_1\Bigl(7r_3(-1)t^{-1}+7r_3(0)+\sum_{j\ge1}r_3(j)7^{\lfloor\frac{7j}{4}\rfloor}t^j\Bigr)\Bigr\},\end{aligned}\end{equation}
where $r_1,r_2,r_3:\mathbb{Z}\to\mathbb{Z}$ range over discrete functions.

The divisibility of the successive functions $L_{k,\alpha}$ can now be expressed in terms of $X_A$ and $X_B$. We state the result below and prove it in the next section using the modular identities from the appendices.

\begin{proposition}\label{prop:L_alpha}
For every integer $\alpha\ge0$, we have
\[
\begin{aligned} L_{k,2\alpha+1}&\in7^\alpha X_B &&(k=4,5),\\[4pt] L_{7,2\alpha+1}&\in7^\alpha qX_B,\\[4pt] L_{k,2\alpha+2}&\in7^\alpha X_A &&(k\in\{4,5,7\}). \end{aligned}
\]
Here $cX:=\{cg:g\in X\}$ denotes multiplication of each element of $X$ by $c$.
\end{proposition}

\section[Proof of the main theorem]{Proof of Theorem~\ref{thm:main}}\label{sec:proof_main}

The modular identities in Appendices~\ref{app:a4}--\ref{app:a7} give seven consecutive initial cases for each of the six families below. We determine the lower summation limits and the powers of $7$ separately for the Laurent polynomials multiplying $1$, $p_0$, and $p_1$. If $bt^n$ is a nonzero term in an initial identity with input exponent $j$, then a lower limit of the form $n\geq\lceil(j+s)/7\rceil$ requires $s\leq 7n-j$. We therefore choose $s$ no larger than the smallest value of $7n-j$ occurring in that component across the seven identities. For the divisibility bound, let $\nu_7(b)$ denote the exponent of $7$ dividing $b$. Writing $b$ as $7^{\lfloor(7n-j+r)/4\rfloor}$ times an integer is possible precisely when $r\leq 4\nu_7(b)+3-7n+j$. The shift $r$ is chosen by taking the smallest of these upper bounds, or a smaller value when a common bound is needed for several values of $k$. Thus both bounds follow from finitely many checks on the appendix identities. The recurrence in Lemma~\ref{lem:U7_recurrence}, together with Lemmas~\ref{lem:order_extension} and~\ref{lem:coefficient_extension}, then extends these expansions, with the same bounds, to every integer $j$.

Consequently, for $k\in\{4,5,7\}$ and $j\in\mathbb{Z}$, there exist discrete functions
\[
c_{k,j,u}^{(h)},\ d_{j,u}^{(h)}:\mathbb{Z}\to\mathbb{Z},
\qquad u,h\in\{0,1,2\},
\]
for which the following expansions hold. The index $u=0,1,2$ corresponds to the input $p_0t^j$, $p_1t^j$, and $t^j$, respectively, while $h=0,1,2$ identifies the component multiplying $1$, $p_0$, and $p_1$. The values of $c_{k,j,u}^{(h)}$ and $d_{j,u}^{(h)}$ are the integers obtained by dividing the coefficients in the expansions for $U_{A_k}$ and $U_7$, respectively, by the corresponding powers of $7$.

\begin{align*}
U_{A_k}(p_0t^j)=&\sum_{n\ge\left\lceil\frac{j+5-2k}{7}\right\rceil}c_{k,j,0}^{(0)}(n)7^{\left\lfloor\frac{7n-j-2}{4}\right\rfloor}t^n+p_0\sum_{n\ge\left\lceil\frac{j+12-2k}{7}\right\rceil}c_{k,j,0}^{(1)}(n)7^{\left\lfloor\frac{7n-j-3}{4}\right\rfloor}t^n\\&+p_1\sum_{n\ge\left\lceil\frac{j+5-2k}{7}\right\rceil}c_{k,j,0}^{(2)}(n)7^{\left\lfloor\frac{7n-j+1}{4}\right\rfloor}t^n,\\
U_{A_k}(p_1t^j)=&\sum_{n\ge\left\lceil\frac{j+5-2k}{7}\right\rceil}c_{k,j,1}^{(0)}(n)7^{\left\lfloor\frac{7n-j-3}{4}\right\rfloor}t^n+p_0\sum_{n\ge\left\lceil\frac{j+12-2k}{7}\right\rceil}c_{k,j,1}^{(1)}(n)7^{\left\lfloor\frac{7n-j-3}{4}\right\rfloor}t^n\\&+p_1\sum_{n\ge\left\lceil\frac{j+5-2k}{7}\right\rceil}c_{k,j,1}^{(2)}(n)7^{\left\lfloor\frac{7n-j}{4}\right\rfloor}t^n,\\
U_{A_k}(t^j)=&\sum_{n\ge\left\lceil\frac{j+4-2k}{7}\right\rceil}c_{k,j,2}^{(0)}(n)7^{\left\lfloor\frac{7n-j-3}{4}\right\rfloor}t^n+p_0\sum_{n\ge\left\lceil\frac{j+11-2k}{7}\right\rceil}c_{k,j,2}^{(1)}(n)7^{\left\lfloor\frac{7n-j-2}{4}\right\rfloor}t^n\\&+p_1\sum_{n\ge\left\lceil\frac{j+4-2k}{7}\right\rceil}c_{k,j,2}^{(2)}(n)7^{\left\lfloor\frac{7n-j+1}{4}\right\rfloor}t^n.
\end{align*}

\noindent Since $U_{B_4}(h)=U_{B_5}(h)=U_{B_7}(qh)=U_7(h)$, the remaining three families take the common form
\begin{align*}
U_7(p_0t^j)=&\sum_{n\ge\left\lceil\frac{j+1}{7}\right\rceil}d_{j,0}^{(0)}(n)7^{\left\lfloor\frac{7n-j}{4}\right\rfloor}t^n+p_0\sum_{n\ge\left\lceil\frac{j+3}{7}\right\rceil}d_{j,0}^{(1)}(n)7^{\left\lfloor\frac{7n-j-1}{4}\right\rfloor}t^n\\&+p_1\sum_{n\ge\left\lceil\frac{j}{7}\right\rceil}d_{j,0}^{(2)}(n)7^{\left\lfloor\frac{7n-j+3}{4}\right\rfloor}t^n,\\
U_7(p_1t^j)=&\sum_{n\ge\left\lceil\frac{j+1}{7}\right\rceil}d_{j,1}^{(0)}(n)7^{\left\lfloor\frac{7n-j-1}{4}\right\rfloor}t^n+p_0\sum_{n\ge\left\lceil\frac{j+4}{7}\right\rceil}d_{j,1}^{(1)}(n)7^{\left\lfloor\frac{7n-j-1}{4}\right\rfloor}t^n\\&+p_1\sum_{n\ge\left\lceil\frac{j}{7}\right\rceil}d_{j,1}^{(2)}(n)7^{\left\lfloor\frac{7n-j+2}{4}\right\rfloor}t^n,\\
U_7(t^j)=&\sum_{n\ge\left\lceil\frac{j}{7}\right\rceil}d_{j,2}^{(0)}(n)7^{\left\lfloor\frac{7n-j-1}{4}\right\rfloor}t^n.
\end{align*}

We are now ready to prove Proposition~\ref{prop:L_alpha}.

\begin{proof}[Proof of Proposition~\ref{prop:L_alpha}]
It suffices to prove the following three assertions:
\[
\begin{aligned}
\text{\textbf{(i)}}\quad &U_{A_k}(1)\in X_B &&(k\in\{4,5,7\}),\\[6pt]
\text{\textbf{(ii)}}\quad &U_7(X_B)\subseteq X_A,\\[6pt]
\text{\textbf{(iii)}}\quad &U_{A_k}(X_A)\subseteq 7X_B &&(k\in\{4,5,7\}).
\end{aligned}
\]
The proposition then follows by induction from the recursive definitions of $L_{k,\alpha}$, using $U_{B_7}(qg)=U_7(g)$ for $k=7$.

\assertionproof{i}
Putting $j=0$ in the general expansion of $U_{A_k}(t^j)$ gives the required divisibility for the coefficients of $t^n$ with $n\geq1$. For the coefficients of $p_0t^n$ and $p_1t^n$, respectively, we use
\[
\left\lfloor\frac{7n-2}{4}\right\rfloor
\geq
\left\lfloor\frac{7n-3}{4}\right\rfloor,
\qquad
\left\lfloor\frac{7n+1}{4}\right\rfloor
\geq
\left\lfloor\frac{7n}{4}\right\rfloor.
\]
The lower summation limits leave only $t^{-1}$, $1$, $p_0$, $p_1t^{-1}$, and $p_1$ to check. Their coefficients are divisible by $7$, as follows from the first identity of Group~III in each of Appendices~\ref{app:a4}--\ref{app:a7}. Thus $U_{A_k}(1)\in X_B$ for every $k\in\{4,5,7\}$.

\assertionproof{ii}
Let $g\in X_B$. With the notation in \eqref{eq:def_XB}, linearity gives
\[
\begin{aligned}
U_7(g)={}&7r_1(-1)U_7(t^{-1})+7r_1(0)+7r_2(0)U_7(p_0)+7r_3(-1)U_7(p_1t^{-1})+7r_3(0)U_7(p_1)\\
&+\sum_{j\geq1}7^{\left\lfloor\frac{7j-3}{4}\right\rfloor}
\bigl(r_1(j)U_7(t^j)+r_2(j)U_7(p_0t^j)\bigr)+\sum_{j\geq1}r_3(j)7^{\left\lfloor\frac{7j}{4}\right\rfloor}
U_7(p_1t^j).
\end{aligned}
\]

We first consider the terms outside the sums. Since $U_7(t^{-1})=-4-7t$, by Group~VI of Appendix~\ref{app:a4},
\[
\begin{aligned}
7r_1(-1)U_7(t^{-1})+7r_1(0)
&=7\bigl(r_1(0)-4r_1(-1)\bigr)-7^2r_1(-1)t\\
&\in X_A.
\end{aligned}
\]
For
\[
7U_7(p_0),\qquad
7U_7(p_1t^{-1}),\qquad
7U_7(p_1),
\]
the corresponding general identities, with $j=0,-1,0$, show that there are no negative powers of $t$, and that the terms at $t^0$ belong to $7\mathbb{Z}+7p_1\mathbb{Z}$. For $n\geq1$, the coefficient valuations are at least $1+\left\lfloor(7n-1)/4\right\rfloor$ in the scalar and $p_0$-parts, and $1+\left\lfloor(7n+2)/4\right\rfloor$ in the $p_1$-part. The inequalities
\[
\begin{aligned}
1+\left\lfloor\frac{7n-1}{4}\right\rfloor
&=\left\lfloor\frac{7n+3}{4}\right\rfloor
\geq\left\lfloor\frac{7n+2}{4}\right\rfloor,\\[4pt]
1+\left\lfloor\frac{7n+2}{4}\right\rfloor
&=\left\lfloor\frac{7n+6}{4}\right\rfloor
\geq\left\lfloor\frac{7n+5}{4}\right\rfloor
\end{aligned}
\]
therefore place all three expressions in $X_A$.

For the terms under the sums, $j\geq1$, and
\[
\left\lfloor\frac{7j}{4}\right\rfloor
\geq\left\lfloor\frac{7j-3}{4}\right\rfloor
=j+\left\lfloor\frac{3j-3}{4}\right\rfloor
\geq j.
\]
Combining this with the general $U_7$-identities, we use
\[
\begin{aligned}
j+\left\lfloor\frac{7n-j-1}{4}\right\rfloor
&=\left\lfloor\frac{7n+3j-1}{4}\right\rfloor
\geq\left\lfloor\frac{7n+2}{4}\right\rfloor,\\[4pt]
j+\left\lfloor\frac{7n-j+2}{4}\right\rfloor
&=\left\lfloor\frac{7n+3j+2}{4}\right\rfloor
\geq\left\lfloor\frac{7n+5}{4}\right\rfloor.
\end{aligned}
\]
The first inequality gives the required divisibility in the scalar and $p_0$-parts, and the second gives it in the $p_1$-part. Since the lower summation limits give $n\geq1$ throughout, both sums belong to $X_A$. Thus $U_7(g)\in X_A$, proving (ii).

\assertionproof{iii}
Now let $g\in X_A$. Separating the terms with $j=0$, we have
\[
\begin{aligned}
U_{A_k}(g)
={}&7r_1(0)U_{A_k}(1)+7r_3(0)U_{A_k}(p_1)\\
&+\sum_{j\geq1}7^{\left\lfloor\frac{7j+2}{4}\right\rfloor}
\bigl(r_1(j)U_{A_k}(t^j)+r_2(j)U_{A_k}(p_0t^j)\bigr)\\
&+\sum_{j\geq1}r_3(j)
7^{\left\lfloor\frac{7j+5}{4}\right\rfloor}
U_{A_k}(p_1t^j).
\end{aligned}
\]
Assertion~(i) gives $U_{A_k}(1)\in X_B$. For $U_{A_k}(p_1)$, the general expansion with $j=0$ gives the required coefficient bounds for $n\geq1$, while the coefficients at nonpositive powers of $t$ are divisible by $7$, by the first identity of Group~II in each of Appendices~\ref{app:a4}--\ref{app:a7}. Hence $U_{A_k}(p_1)\in X_B$, and the first two terms belong to $7X_B$.

For $j\geq1$, we have
\[
\left\lfloor\frac{7j+5}{4}\right\rfloor
\geq\left\lfloor\frac{7j+2}{4}\right\rfloor
=j+1+\left\lfloor\frac{3j-2}{4}\right\rfloor
\geq j+1.
\]
In each of the expansions of $U_{A_k}(t^j)$, $U_{A_k}(p_0t^j)$, and $U_{A_k}(p_1t^j)$, the coefficient valuations are at least $\left\lfloor(7n-j-3)/4\right\rfloor$ in the scalar and $p_0$-parts, and $\left\lfloor(7n-j)/4\right\rfloor$ in the $p_1$-part. For $j,n\geq1$,
\[
\begin{aligned}
\left\lfloor\frac{7j+2}{4}\right\rfloor
+\left\lfloor\frac{7n-j-3}{4}\right\rfloor
&\geq1+\left\lfloor\frac{7n+3j-3}{4}\right\rfloor\\
&\geq1+\left\lfloor\frac{7n-3}{4}\right\rfloor,\\[6pt]
\left\lfloor\frac{7j+2}{4}\right\rfloor
+\left\lfloor\frac{7n-j}{4}\right\rfloor
&\geq1+\left\lfloor\frac{7n+3j}{4}\right\rfloor\\
&\geq1+\left\lfloor\frac{7n}{4}\right\rfloor.
\end{aligned}
\]
Since $\left\lfloor(7j+5)/4\right\rfloor\geq\left\lfloor(7j+2)/4\right\rfloor$, these bounds also hold for the terms in the last sum. Thus the coefficients of $t^n$, $p_0t^n$, and $p_1t^n$ have the divisibility required in $7X_B$ for every $n\geq1$.

For the nonpositive powers, the lower summation limits allow only the five terms considered in (i). Their coefficients are integers by Lemma~\ref{lem:order_extension}. Both factors $7^{\left\lfloor(7j+2)/4\right\rfloor}$ and $7^{\left\lfloor(7j+5)/4\right\rfloor}$ are divisible by $7^2$ for $j\geq1$, so the coefficients of these five terms also have the required divisibility.

Both sums therefore belong to $7X_B$. Together with the first two terms, this gives $U_{A_k}(g)\in7X_B$, proving (iii). 
\par\medskip
\noindent
We conclude by induction on $\alpha$. The case $\alpha=0$ follows from assertions~(i) and~(ii) and the recursive definitions in Section~\ref{sec:auxiliary_functions}. Suppose that, for some $\alpha\geq0$,
\[
L_{k,2\alpha+2}\in7^\alpha X_A.
\]
By assertion~(iii), these recursions give
\[
\begin{aligned}
L_{k,2\alpha+3}&\in7^{\alpha+1}X_B &&(k=4,5),\\
L_{7,2\alpha+3}&\in7^{\alpha+1}qX_B.
\end{aligned}
\]
Applying assertion~(ii) at the next even step, we obtain
\[
L_{k,2\alpha+4}\in7^{\alpha+1}X_A.
\]
This completes the induction and proves Proposition~\ref{prop:L_alpha}.
\end{proof}

\begin{proof}[Proof of Theorem~\ref{thm:main}]
By \eqref{eq:def_t} and \eqref{eq:def_p0_p1}, the Laurent expansions of $t$, $t^{-1}$, and $p_0$ have integer coefficients. Also,
\[
f_1^7\equiv f_7\pmod7,\qquad
f_2^7\equiv f_{14}\pmod7
\quad\Longrightarrow\quad
\frac{f_{14}f_1^7}{f_7f_2^7}\equiv1\pmod7,
\]
so $p_1$ has integer coefficients. Hence \eqref{eq:def_XA} and \eqref{eq:def_XB}, together with Proposition~\ref{prop:L_alpha}, give
\[
L_{k,2\alpha-1}\equiv L_{k,2\alpha}\equiv0\pmod{7^\alpha}
\qquad(\alpha\geq1),
\]
where the congruences hold coefficientwise.

Since $F_k(q)$ has integer coefficients by \eqref{eq:gf_ak}, the generating-function identities in Section~\ref{sec:auxiliary_functions} yield
\[
\begin{aligned}
\sum_{n=0}^{\infty}
a_k\bigl(7^{2\alpha-1}n+\lambda_{k,2\alpha-1}\bigr)q^n
&=F_k(q^7)L_{k,2\alpha-1}
\equiv0\pmod{7^\alpha},\\[6pt]
\sum_{n=0}^{\infty}
a_k\bigl(7^{2\alpha}n+\lambda_{k,2\alpha}\bigr)q^n
&=F_k(q)L_{k,2\alpha}
\equiv0\pmod{7^\alpha}.
\end{aligned}
\]
Comparing coefficients gives the odd and even cases of Theorem~\ref{thm:main}, respectively.
\end{proof}

\section*{Appendices}

The identities in Appendices~\ref{app:a4}--\ref{app:a7} were obtained with Garvan's \texttt{qseries} package and proved with his \texttt{ETA} package in \texttt{Maple}~\cite{GarvanQseries,GarvanETA}. The packages and their documentation are available at \url{https://qseries.org/fgarvan/qmaple/qseries/} and  \url{https://qseries.org/fgarvan/qmaple/ETA/}.

For each identity, we first computed the $q$-expansion of the left-hand side using \eqref{eq:Up_definition}. We then used \texttt{\textbf{findlincombo}} to find a relation of the form
\[
P(t)+p_0Q(t)+p_1R(t),
\qquad P,Q,R\in\mathbb{Z}[t,t^{-1}].
\]
The coefficients of these Laurent polynomials were determined by comparing the computed expansion with those of $t^n$, $p_0t^n$, and $p_1t^n$ over a finite range of $n$. The range of exponents and the number of terms in the expansions were increased as needed. Where negative powers of $t$ occurred, we first multiplied by a suitable power of $t$ and divided by the same power after finding the relation. This procedure gave the candidate identities.

Throughout the appendices, each nonzero integer coefficient $b$ is written as $7^{\nu_7(b)}b_0$, where $7\nmid b_0$. This notation makes the powers of $7$ explicit; the integer $b_0$ is left unfactored.

To prove these identities, we used the valence-formula method described by Chen, Chen and Garvan~\cite[Section~2.2]{CCG24}, together with its application to $U_p$ identities in Garvan's tutorial~\cite[Section~3.2]{GarvanETA}. After substituting the eta-quotient expressions from Section~\ref{sec:auxiliary_functions} and expanding the terms involving $p_1$, we wrote each identity as $U_7(G)=H$, where $G$ is a linear combination of eta quotients on $\Gamma_0(98)$ and $H$ is a linear combination of eta quotients on $\Gamma_0(14)$. By Theorem~\ref{thm:Up_level}, with $p=7$ and $N=14$, the function $U_7(G)$ is modular on $\Gamma_0(14)$. Thus $U_7(G)-H$ is a modular function on this group.

We then used \texttt{\textbf{provemodfuncGAMMA0UpETAidBATCH}} with $p=7$ and $N=14$ to verify each identity. The command checks modularity and determines lower bounds for the orders at the cusps. By the valence formula, these bounds give a sufficient order of vanishing at infinity for $U_7(G)-H$ to be identically zero. The required vanishing was checked from its $q$-expansion in every case.

Where the verification required multiplication by a power of $t$, we used \eqref{eq:Up_product} in the form
\[
U_7\bigl(t(7\tau)^dG(\tau)\bigr)
=t(\tau)^dU_7\bigl(G(\tau)\bigr).
\]
Since $t$ has no zeros in $\mathbb{H}$, cancellation of this factor gives the original identity.

The \texttt{Maple} commands used to discover and verify the identities listed in the appendices are available as supplementary material \href{https://github.com/manjilsaikia/manjilsaikia.github.io/tree/da5252100f5a151b2eb6fa5142c96272c1ef70fd/publ/maple}{at this link}.

\appendix
\section[Fundamental relations for a4(n)]{Fundamental relations for \texorpdfstring{$a_4(n)$}{a4(n)}}\label{app:a4}

The six groups below give the identities for $U_{A_4}$ and $U_{B_4}$ applied to $p_0t^j$, $p_1t^j$, and $t^j$.

\par\bigskip
\noindent\textbf{Group I}\par\nobreak\medskip

\auxidentity{
U_{A_4}(p_0)=7^2t^2p_0+7t,
}

\auxidentity{
U_{A_4}(p_0t^{-1})=-7tp_0+ p_1-7t-1,
}

\auxidentity{
U_{A_4}(p_0t^{-2})=p_0\bigl(7^5t^3+2\cdot7^4t^2+54\cdot7t\bigr)+p_1\bigl(-7^4t^2-13\cdot7^2t-48\bigr)+7^4t^2+95\cdot7t+48,
}

\auxidentity{
U_{A_4}(p_0t^{-3})=p_0\bigl(7^7t^4+2\cdot7^5t^3-13\cdot7^4t^2-459\cdot7t\bigr)+p_1\bigl(-7^6t^3-7^4t^2+87\cdot7^2t+58\cdot7\bigr)+5\cdot7^5t^3-3\cdot7^3t^2-90\cdot7^2t-58\cdot7,
}

\auxidentity{
U_{A_4}(p_0t^{-4})=p_0\bigl(-7^9t^5-36\cdot7^7t^4-38\cdot7^6t^3-181\cdot7^3t^2+363\cdot7^2t+8\bigr)+p_1\bigl(7^8t^4+5\cdot7^7t^3+34\cdot7^5t^2+16\cdot7^3t-321\cdot7-t^{-1}\bigr)-7^8t^4-31\cdot7^6t^3-218\cdot7^4t^2-95\cdot7^2t+321\cdot7+t^{-1},
}

\auxidentity{
U_{A_4}(p_0t^{-5})=p_0\bigl(7^{11}t^6+41\cdot7^9t^5+734\cdot7^7t^4+727\cdot7^6t^3+1602\cdot7^4t^2-1025\cdot7^2t-160\bigr)+p_1\bigl(-7^{10}t^5-40\cdot7^8t^4-100\cdot7^7t^3-659\cdot7^5t^2-1374\cdot7^3t+913\cdot7+20t^{-1}\bigr)+6\cdot7^9t^5+265\cdot7^7t^4+93\cdot7^7t^3+4404\cdot7^4t^2+9519\cdot7^2t-6390-20t^{-1},
}

\auxidentity{
U_{A_4}(p_0t^{-6})=p_0\bigl(-7^{13}t^7-46\cdot7^{11}t^6-947\cdot7^9t^5-11264\cdot7^7t^4-9647\cdot7^6t^3-23659\cdot7^4t^2-6634\cdot7^2t+2024\bigr)+p_1\bigl(7^{12}t^6+45\cdot7^{10}t^5+907\cdot7^8t^4+216\cdot7^8t^3+8747\cdot7^5t^2+20686\cdot7^3t+5744\cdot7-253t^{-1}\bigr)-5\cdot7^{11}t^6-258\cdot7^9t^5-828\cdot7^8t^4-1423\cdot7^7t^3-59252\cdot7^4t^2-144218\cdot7^2t-821\cdot7^2+253t^{-1}.
}

\par\bigskip
\noindent\textbf{Group II}\par\nobreak\medskip

\auxidentity{
U_{A_4}(p_1)=p_0\bigl(8\cdot 7^{11}t^7+144\cdot 7^9t^6+139\cdot 7^8t^5+404\cdot 7^6t^4+417\cdot 7^4t^3-15\cdot 7^2t^2-8\cdot 7t\bigr)+p_1\bigl(7^{14}t^8+27\cdot 7^{12}t^7+293\cdot 7^{10}t^6+1604\cdot 7^8t^5+646\cdot 7^7t^4+845\cdot 7^5t^3+8\cdot 7^5t^2+114\cdot 7t+7\bigr)+7^{13}t^8+19\cdot 7^{11}t^7+141\cdot 7^9t^6+69\cdot 7^8t^5+99\cdot 7^6t^4-12\cdot 7^4t^3-297\cdot 7^2t^2-104\cdot 7t-7,
}

\auxidentity{
U_{A_4}(p_1t^{-1})=-8\cdot 7tp_0+\bigl(7^2t+7\bigr)p_1-7,
}

\auxidentity{
U_{A_4}(p_1t^{-2})=p_0\bigl(3\cdot 7^3t^2+11\cdot 7t\bigr)+p_1\bigl(-3\cdot 7^2t-9\bigr)+7^3t^2+18\cdot 7t+9,
}

\auxidentity{
U_{A_4}(p_1t^{-3})=p_0\bigl(2\cdot 7^5t^3+27\cdot 7^3t^2+131\cdot 7t\bigr)+p_1\bigl(-2\cdot 7^4t^2-24\cdot 7^2t-118\bigr)-7^5t^3+4\cdot 7^3t^2+25\cdot 7^2t+118,
}

\auxidentity{
U_{A_4}(p_1t^{-4})=p_0\bigl(2\cdot 7^7t^4-256\cdot 7^3t^2-206\cdot 7^2t-8\bigr)+p_1\bigl(-2\cdot 7^6t^3+3\cdot 7^4t^2+241\cdot 7^2t+183\cdot 7+t^{-1}\bigr)-7^7t^4+4\cdot 7^5t^3-23\cdot 7^3t^2-247\cdot 7^2t-183\cdot 7-t^{-1},
}

\auxidentity{
U_{A_4}(p_1t^{-5})=p_0\bigl(-5\cdot 7^9t^5-138\cdot 7^7t^4-143\cdot 7^6t^3-171\cdot 7^4t^2+1192\cdot 7^2t+152\bigr)+p_1\bigl(5\cdot 7^8t^4+19\cdot 7^7t^3+127\cdot 7^5t^2+130\cdot 7^3t-1058\cdot 7-19t^{-1}\bigr)-22\cdot 7^7t^4-106\cdot 7^6t^3-814\cdot 7^4t^2-862\cdot 7^2t+7407+19t^{-1},
}

\auxidentity{
U_{A_4}(p_1t^{-6})=p_0\bigl(8\cdot 7^{11}t^6+38\cdot 7^{10}t^5+3726\cdot 7^7t^4+3298\cdot 7^6t^3+7423\cdot 7^4t^2-2693\cdot 7^2t-1864\bigr)+p_1\bigl(-8\cdot 7^{10}t^5-37\cdot 7^9t^4-72\cdot 7^8t^3-2978\cdot 7^5t^2-6404\cdot 7^3t+2431\cdot 7+233t^{-1}\bigr)+7^{11}t^6+68\cdot 7^9t^5+1727\cdot 7^7t^4+3238\cdot 7^6t^3+19876\cdot 7^4t^2+907\cdot 7^4t-17029-233t^{-1}.
}

\par\bigskip
\noindent\textbf{Group III}\par\nobreak\medskip

\auxidentity{
U_{A_4}(1)=p_0\bigl(3\cdot 7^8t^5+52\cdot 7^6t^4+321\cdot 7^4t^3+111\cdot 7^3t^2+72\cdot 7t\bigr)+p_1\bigl(-3\cdot 7^7t^4-50\cdot 7^5t^3-297\cdot 7^3t^2-99\cdot 7^2t-9\cdot 7\bigr)+7^8t^5+34\cdot 7^6t^4+405\cdot 7^4t^3+2150\cdot 7^2t^2+695\cdot 7t+9\cdot 7,
}

\auxidentity{
U_{A_4}(t^{-1})=-7t,
}

\auxidentity{
U_{A_4}(t^{-2})=p_0\bigl(3\cdot 7^3t^2+11\cdot 7t\bigr)+p_1\bigl(-3\cdot 7^2t-10\bigr)+7^3t^2+26\cdot 7t+10,
}

\auxidentity{
U_{A_4}(t^{-3})=p_0\bigl(-6\cdot 7^5t^3-85\cdot 7^3t^2-5\cdot 7^3t+8\bigr)+p_1\bigl(6\cdot 7^4t^2+80\cdot 7^2t+31\cdot 7-t^{-1}\bigr)-7^5t^3-52\cdot 7^3t^2-585\cdot 7t-31\cdot 7+t^{-1},
}

\auxidentity{
U_{A_4}(t^{-4})=p_0\bigl(-6\cdot 7^7t^4-8\cdot 7^5t^3+80\cdot 7^4t^2+306\cdot 7^2t-96\bigr)+p_1\bigl(6\cdot 7^6t^3+4\cdot 7^4t^2-76\cdot 7^3t-270\cdot 7+12t^{-1}\bigr)-7^7t^4-36\cdot 7^5t^3+50\cdot 7^3t^2+550\cdot 7^2t+270\cdot 7-12t^{-1},
}

\auxidentity{
U_{A_4}(t^{-5})=p_0\bigl(3\cdot 7^9t^5+158\cdot 7^7t^4+23\cdot 7^7t^3-67\cdot 7^4t^2-1872\cdot 7^2t+808\bigr)+p_1\bigl(-3\cdot 7^8t^4-22\cdot 7^7t^3-145\cdot 7^5t^2+92\cdot 7^3t+1650\cdot 7-101t^{-1}\bigr)+34\cdot 7^7t^4+149\cdot 7^6t^3+930\cdot 7^4t^2-732\cdot 7^2t-11549+101t^{-1},
}

\auxidentity{
U_{A_4}(t^{-6})=p_0\bigl(-10\cdot 7^{10}t^5-50\cdot 7^9t^4-2822\cdot 7^6t^3-5545\cdot 7^4t^2+7179\cdot 7^2t-800\cdot 7\bigr)+p_1\bigl(68\cdot 7^8t^4+334\cdot 7^7t^3+2560\cdot 7^5t^2+4721\cdot 7^3t-6336\cdot 7+100\cdot 7t^{-1}\bigr)+7^{11}t^6+20\cdot 7^9t^5-442\cdot 7^7t^4-2236\cdot 7^6t^3-17218\cdot 7^4t^2-32640\cdot 7^2t+44332-100\cdot 7t^{-1}.
}

\par\bigskip
\noindent\textbf{Group IV}\par\nobreak\medskip

\auxidentity{
U_{B_4}(p_0)=p_0\bigl(8\cdot 7^7t^4+80\cdot 7^5t^3+216\cdot 7^3t^2+79\cdot 7t\bigr)+p_1\bigl(7^{10}t^5+19\cdot 7^8t^4+18\cdot 7^7t^3+327\cdot 7^4t^2+34\cdot 7^3t+4\bigr)+7^9t^5+11\cdot 7^7t^4+38\cdot 7^5t^3+31\cdot 7^3t^2+6\cdot 7t,
}

\auxidentity{
U_{B_4}(p_0t^{-1})=p_0\bigl(-7^3t^2-8\cdot 7t\bigr)+p_1\bigl(7^2t+7\bigr)+1-7^2t,
}

\auxidentity{
U_{B_4}(p_0t^{-2})=p_0\bigl(-7^5t^3-8\cdot 7^3t^2\bigr)+p_1\bigl(7^4t^2+7^3t\bigr)-7^4t^2-7^3t-12,
}

\auxidentity{
U_{B_4}(p_0t^{-3})=p_0\bigl(1+6\cdot 7^7t^4+62\cdot 7^5t^3+27\cdot 7^4t^2+10\cdot 7^2t\bigr)+p_1\bigl(-6\cdot 7^6t^3-8\cdot 7^5t^2-3\cdot 7^4t-4\cdot 7\bigr)+2\cdot 7^6t^3+24\cdot 7^4t^2+83\cdot 7^2t+108,
}

\auxidentity{
U_{B_4}(p_0t^{-4})=p_0\bigl(-7^9t^5-50\cdot 7^7t^4-60\cdot 7^6t^3-162\cdot 7^4t^2-60\cdot 7^2t-6\bigr)+p_1\bigl(7^8t^4+7^8t^3+54\cdot 7^5t^2+18\cdot 7^4t+24\cdot 7\bigr)-7^8t^4-25\cdot 7^6t^3-186\cdot 7^4t^2-498\cdot 7^2t-120\cdot 7,
}

\auxidentity{
U_{B_4}(p_0t^{-5})=p_0\bigl(-7^{11}t^6-8\cdot 7^9t^5+3\cdot 7^9t^4+30\cdot 7^7t^3+81\cdot 7^5t^2+30\cdot 7^3t+3\cdot 7\bigr)+p_1\bigl(7^{10}t^5+7^9t^4-3\cdot 7^8t^3-27\cdot 7^6t^2-9\cdot 7^5t-12\cdot 7^2\bigr)-7^{10}t^5-7^9t^4+9\cdot 7^7t^3+93\cdot 7^5t^2+249\cdot 7^3t+836\cdot 7,
}

\auxidentity{
U_{B_4}(p_0t^{-6})=p_0\bigl(-7^{13}t^7-8\cdot 7^{11}t^6+2\cdot 7^9t^4+20\cdot 7^7t^3+54\cdot 7^5t^2+20\cdot 7^3t+2\cdot 7\bigr)+p_1\bigl(7^{12}t^6+7^{11}t^5-2\cdot 7^8t^3-18\cdot 7^6t^2-6\cdot 7^5t-8\cdot 7^2\bigr)-7^{12}t^6-7^{11}t^5+6\cdot 7^7t^3+62\cdot 7^5t^2+166\cdot 7^3t-748\cdot 7^2.
}

\par\bigskip
\noindent\textbf{Group V}\par\nobreak\medskip

\auxidentity{
U_{B_4}(p_1t^{-1})=p_0\bigl(8\cdot 7^8t^4+80\cdot 7^6t^3+216\cdot 7^4t^2+552\cdot 7t\bigr)+p_1\bigl(7^{11}t^5+19\cdot 7^9t^4+18\cdot 7^8t^3+327\cdot 7^5t^2+34\cdot 7^4t+29\bigr)-1+7^{10}t^5+11\cdot 7^8t^4+38\cdot 7^6t^3+31\cdot 7^4t^2+41\cdot 7t,
}

\auxidentity{
U_{B_4}(p_1t^{-2})=p_0\bigl(-8\cdot 7^3t^2-8\cdot 7^2t\bigr)+p_1\bigl(8\cdot 7^2t+7^2\bigr)-7^3t^2-8\cdot 7^2t+11,
}

\auxidentity{
U_{B_4}(p_1t^{-3})=p_0\bigl(-8\cdot 7^5t^3-8\cdot 7^4t^2\bigr)+p_1\bigl(8\cdot 7^4t^2+7^4t\bigr)-7^5t^3-8\cdot 7^4t^2-7^4t-96,
}

\auxidentity{
U_{B_4}(p_1t^{-4})=p_0\bigl(48\cdot 7^7t^4+72\cdot 7^6t^3+216\cdot 7^4t^2+80\cdot 7^2t+8\bigr)+p_1\bigl(-48\cdot 7^6t^3-65\cdot 7^5t^2-24\cdot 7^4t-32\cdot 7\bigr)-7^7t^4+16\cdot 7^6t^3+199\cdot 7^4t^2+664\cdot 7^2t+740,
}

\auxidentity{
U_{B_4}(p_1t^{-5})=p_0\bigl(-8\cdot 7^9t^5-64\cdot 7^8t^4-80\cdot 7^7t^3-216\cdot 7^5t^2-80\cdot 7^3t-8\cdot 7\bigr)+p_1\bigl(8\cdot 7^8t^4+9\cdot 7^8t^3+72\cdot 7^6t^2+24\cdot 7^5t+32\cdot 7^2\bigr)-7^9t^5-8\cdot 7^8t^4-31\cdot 7^7t^3-248\cdot 7^5t^2-664\cdot 7^3t-740\cdot 7,
}

\auxidentity{
U_{B_4}(p_1t^{-6})=p_0\bigl(-8\cdot 7^{11}t^6-8\cdot 7^{10}t^5+40\cdot 7^9t^4+400\cdot 7^7t^3+1080\cdot 7^5t^2+400\cdot 7^3t+40\cdot 7\bigr)+p_1\bigl(8\cdot 7^{10}t^5+7^{10}t^4-40\cdot 7^8t^3-360\cdot 7^6t^2-120\cdot 7^5t-160\cdot 7^2\bigr)-7^{11}t^6-8\cdot 7^{10}t^5-7^{10}t^4+120\cdot 7^7t^3+1240\cdot 7^5t^2+3320\cdot 7^3t+4720\cdot 7,
}

\auxidentity{
U_{B_4}(p_1t^{-7})=p_0\bigl(-8\cdot 7^{13}t^7-8\cdot 7^{12}t^6-136\cdot 7^9t^4-1360\cdot 7^7t^3-3672\cdot 7^5t^2-1360\cdot 7^3t-136\cdot 7\bigr)+p_1\bigl(8\cdot 7^{12}t^6+6\cdot 7^{11}t^5-19\cdot 7^9t^4+818\cdot 7^7t^3+1167\cdot 7^6t^2+2798\cdot 7^4t+11\cdot 7^4+t^{-1}\bigr)+6\cdot 7^{13}t^7+20\cdot 7^{12}t^6+272\cdot 7^{10}t^5+1893\cdot 7^8t^4+435\cdot 7^7t^3-2953\cdot 7^5t^2-10620\cdot 7^3t-27029\cdot 7.
}

\par\bigskip
\noindent\textbf{Group VI}\par\nobreak\medskip

\auxidentity{
U_{B_4}(1)=1,}

\auxidentity{
U_{B_4}(t^{-1})=-7t-4,}

\auxidentity{
U_{B_4}(t^{-2})=-7^3t^2+20,}

\auxidentity{
U_{B_4}(t^{-3})=-7^5t^3-88,}

\auxidentity{
U_{B_4}(t^{-4})=-7^7t^4+260,}

\auxidentity{
U_{B_4}(t^{-5})=-7^9t^5+68\cdot 7,}

\auxidentity{
U_{B_4}(t^{-6})=-7^{11}t^6-2392\cdot 7.}

\setcounter{section}{1}
\section[Fundamental relations for a5(n)]{Fundamental relations for \texorpdfstring{$a_5(n)$}{a5(n)}}\label{app:a5}

The six groups below give the identities for $U_{A_5}$ and $U_{B_5}$ applied to $p_0t^j$, $p_1t^j$, and $t^j$.

\par\bigskip
\noindent\textbf{Group I}\par\nobreak\medskip

\auxidentity{U_{A_5}(p_0)=p_0\bigl(7^4t^3+15\cdot7^2t^2+8\cdot7t\bigr)+p_1\bigl(-7^3t^2-2\cdot7^2t-7\bigr)+7^3t^2+2\cdot7^2t+7,}

\auxidentity{U_{A_5}(p_0t^{-1})=p_0\bigl(-7^3t^2-8\cdot7t\bigr)+p_1\bigl(7^2t+7\bigr)-7^2t-7,}

\auxidentity{U_{A_5}(p_0t^{-2})=p_0\bigl(7^4t^2+9\cdot7^2t+8\bigr)+p_1\bigl(-7^3t-8\cdot7-t^{-1}\bigr)+7^3t+8\cdot7+t^{-1},}

\auxidentity{U_{A_5}(p_0t^{-3})=p_0\bigl(7^7t^4+19\cdot7^5t^3+11\cdot7^4t^2-165\cdot7t-88\bigr)+p_1\bigl(-7^6t^3-18\cdot7^4t^2-66\cdot7^2t+22\cdot7+11t^{-1}\bigr)+7^6t^3+18\cdot7^4t^2+66\cdot7^2t-22\cdot7-11t^{-1},}

\auxidentity{U_{A_5}(p_0t^{-4})=p_0\bigl(7^9t^5+8\cdot7^7t^4-90\cdot7^5t^3-860\cdot7^3t^2-10\cdot7^3t+720\bigr)+p_1\bigl(-7^8t^4-7^7t^3+90\cdot7^4t^2+110\cdot7^3t+50\cdot7-90t^{-1}\bigr)+7^8t^4+7^7t^3-90\cdot7^4t^2-110\cdot7^3t-50\cdot7+90t^{-1},}

\auxidentity{U_{A_5}(p_0t^{-5})=p_0\bigl(-7^{11}t^6-46\cdot7^9t^5-570\cdot7^7t^4-323\cdot7^6t^3+57\cdot7^4t^2+1045\cdot7^2t-5016\bigr)+p_1\bigl(7^{10}t^5+45\cdot7^8t^4+76\cdot7^7t^3+285\cdot7^5t^2-76\cdot7^3t-836\cdot7+627t^{-1}\bigr)-7^{10}t^5-45\cdot7^8t^4-76\cdot7^7t^3-285\cdot7^5t^2+76\cdot7^3t+836\cdot7-627t^{-1},}

\auxidentity{U_{A_5}(p_0t^{-6})=p_0\bigl(7^{13}t^7+54\cdot7^{11}t^6+1242\cdot7^9t^5+13110\cdot7^7t^4+8947\cdot7^6t^3+2139\cdot7^5t^2-1541\cdot7^2t+30360\bigr)+p_1\bigl(-7^{12}t^6-53\cdot7^{10}t^5-1196\cdot7^8t^4-1748\cdot7^7t^3-8073\cdot7^5t^2-13018\cdot7^3t+874\cdot7-3795t^{-1}\bigr)+7^{12}t^6+53\cdot7^{10}t^5+1196\cdot7^8t^4+1748\cdot7^7t^3+8073\cdot7^5t^2+13018\cdot7^3t-874\cdot7+3795t^{-1}.}

\par\bigskip
\noindent\textbf{Group II}\par\nobreak\medskip

\auxidentity{U_{A_5}(p_1)=p_0\bigl(-4\cdot7^9t^6-12\cdot7^8t^5-108\cdot7^6t^4-480\cdot7^4t^3-984\cdot7^2t^2-88\cdot7t\bigr)+p_1\bigl(7^{10}t^6+22\cdot7^8t^5+29\cdot7^7t^4+152\cdot7^5t^3+492\cdot7^3t^2+884\cdot7t+11\cdot7\bigr)-5\cdot7^8t^5-99\cdot7^6t^4-809\cdot7^4t^3-3254\cdot7^2t^2-883\cdot7t-11\cdot7,}

\auxidentity{U_{A_5}(p_1t^{-1})= 7t,}

\auxidentity{U_{A_5}(p_1t^{-2})=p_0\bigl(-8\cdot7^3t^2-64\cdot7t-8\bigr)+p_1\bigl(8\cdot7^2t+57+t^{-1}\bigr)-7^3t^2-58\cdot7t-57-t^{-1},}

\auxidentity{U_{A_5}(p_1t^{-3})=p_0\bigl(4\cdot7^5t^3+92\cdot7^3t^2+540\cdot7t+80\bigr)+p_1\bigl(-4\cdot7^4t^2-89\cdot7^2t-482-10t^{-1}\bigr)+7^5t^3+38\cdot7^3t^2+633\cdot7t+482+10t^{-1},}

\auxidentity{U_{A_5}(p_1t^{-4})=p_0\bigl(-20\cdot7^5t^3-484\cdot7^3t^2-436\cdot7^2t-632\bigr)+p_1\bigl(3\cdot7^5t^2+472\cdot7^2t+8\cdot7^3+79t^{-1}\bigr)-7^7t^4-18\cdot7^5t^3-237\cdot7^3t^2-480\cdot7^2t-8\cdot7^3-79t^{-1},}

\auxidentity{U_{A_5}(p_1t^{-5})=p_0\bigl(8\cdot7^7t^4+20\cdot7^6t^3+292\cdot7^4t^2+1964\cdot7^2t+4360\bigr)+p_1\bigl(-7^7t^3-20\cdot7^5t^2-286\cdot7^3t-1790\cdot7-545t^{-1}\bigr)-7^9t^5+25\cdot7^6t^3+230\cdot7^4t^2+2046\cdot7^2t+1790\cdot7+545t^{-1},}

\auxidentity{U_{A_5}(p_1t^{-6})=p_0\bigl(-4\cdot7^{11}t^6-116\cdot7^9t^5-1460\cdot7^7t^4-176\cdot7^7t^3-3560\cdot7^4t^2-7992\cdot7^2t-26624\bigr)+p_1\bigl(4\cdot7^{10}t^5+16\cdot7^9t^4+194\cdot7^7t^3+1123\cdot7^5t^2+3226\cdot7^3t+7425\cdot7+3328t^{-1}\bigr)-11\cdot7^9t^5-680\cdot7^7t^4-1433\cdot7^6t^3-8391\cdot7^4t^2-22809\cdot7^2t-51974-3328t^{-1}.}

\par\bigskip
\noindent\textbf{Group III}\par\nobreak\medskip

\auxidentity{U_{A_5}(1)=p_0\bigl(4\cdot7^{10}t^6+92\cdot7^8t^5+828\cdot7^6t^4+3552\cdot7^4t^3+992\cdot7^3t^2+632\cdot7t\bigr)+p_1\bigl(-4\cdot7^9t^5-89\cdot7^7t^4-776\cdot7^5t^3-3231\cdot7^3t^2-880\cdot7^2t-79\cdot7\bigr)+7^{10}t^6+45\cdot7^8t^5+727\cdot7^6t^4+5687\cdot7^4t^3+22807\cdot7^2t^2+6161\cdot7t+79\cdot7,}

\auxidentity{U_{A_5}(t^{-1})=p_0\bigl(8\cdot7t+8\bigr)+p_1\bigl(-8-t^{-1}\bigr)+7t+8+t^{-1},}

\auxidentity{U_{A_5}(t^{-2})=p_0\bigl(-8\cdot7t-24\bigr)+p_1\bigl(10+3t^{-1}\bigr)-7^3t^2-2\cdot7t-10-3t^{-1},}

\auxidentity{U_{A_5}(t^{-3})=p_0\bigl(4\cdot7^5t^3+4\cdot7^4t^2-4\cdot7^2t+80\bigr)+p_1\bigl(-4\cdot7^4t^2-24\cdot7^2t+2\cdot7-10t^{-1}\bigr)+7^5t^3+38\cdot7^3t^2+178\cdot7t-2\cdot7+10t^{-1},}

\auxidentity{U_{A_5}(t^{-4})=p_0\bigl(-8\cdot7^7t^4-172\cdot7^5t^3-20\cdot7^5t^2-12\cdot7^3t-88\bigr)+p_1\bigl(8\cdot7^6t^3+164\cdot7^4t^2+124\cdot7^3t+76\cdot7+11t^{-1}\bigr)-7^7t^4-74\cdot7^5t^3-1238\cdot7^3t^2-876\cdot7^2t-76\cdot7-11t^{-1},}

\auxidentity{U_{A_5}(t^{-5})=p_0\bigl(-8\cdot7^9t^5-48\cdot7^7t^4+20\cdot7^7t^3+1196\cdot7^4t^2+1044\cdot7^2t-2024\bigr)+p_1\bigl(8\cdot7^8t^4+6\cdot7^7t^3-136\cdot7^5t^2-1062\cdot7^3t-886\cdot7+253t^{-1}\bigr)-7^9t^5-8\cdot7^8t^4-24\cdot7^6t^3+1042\cdot7^4t^2+7478\cdot7^2t+886\cdot7-253t^{-1},}

\auxidentity{U_{A_5}(t^{-6})=p_0\bigl(4\cdot7^{11}t^6+260\cdot7^9t^5+452\cdot7^8t^4+1312\cdot7^6t^3-3256\cdot7^4t^2-4664\cdot7^2t+3952\cdot7\bigr)+p_1\bigl(-4\cdot7^{10}t^5-257\cdot7^8t^4-426\cdot7^7t^3-1157\cdot7^5t^2+2882\cdot7^3t+3655\cdot7-494\cdot7t^{-1}\bigr)+45\cdot7^9t^5+1903\cdot7^7t^4+2907\cdot7^6t^3+7569\cdot7^4t^2-20401\cdot7^2t-25584+494\cdot7t^{-1}.}

\par\bigskip
\noindent\textbf{Group IV}\par\nobreak\medskip

\auxidentity{U_{B_5}(p_0)=p_0\bigl(8\cdot7^7t^4+80\cdot7^5t^3+216\cdot7^3t^2+79\cdot7t\bigr)+p_1\bigl(7^{10}t^5+19\cdot7^8t^4+18\cdot7^7t^3+327\cdot7^4t^2+34\cdot7^3t+4\bigr)+7^9t^5+11\cdot7^7t^4+38\cdot7^5t^3+31\cdot7^3t^2+6\cdot7t,}

\auxidentity{U_{B_5}(p_0t^{-1})=p_0\bigl(-7^3t^2-8\cdot7t\bigr)+p_1\bigl(7^2t+7\bigr)-7^2t+1,}

\auxidentity{U_{B_5}(p_0t^{-2})=p_0\bigl(-7^5t^3-8\cdot7^3t^2\bigr)+p_1\bigl(7^4t^2+7^3t\bigr)-7^4t^2-7^3t-12,}

\auxidentity{U_{B_5}(p_0t^{-3})=p_0\bigl(6\cdot7^7t^4+62\cdot7^5t^3+27\cdot7^4t^2+10\cdot7^2t+1\bigr)+p_1\bigl(-6\cdot7^6t^3-8\cdot7^5t^2-3\cdot7^4t-4\cdot7\bigr)+2\cdot7^6t^3+24\cdot7^4t^2+83\cdot7^2t+108,}

\auxidentity{U_{B_5}(p_0t^{-4})=p_0\bigl(-7^9t^5-50\cdot7^7t^4-60\cdot7^6t^3-162\cdot7^4t^2-60\cdot7^2t-6\bigr)+p_1\bigl(7^8t^4+7^8t^3+54\cdot7^5t^2+18\cdot7^4t+24\cdot7\bigr)-7^8t^4-25\cdot7^6t^3-186\cdot7^4t^2-498\cdot7^2t-120\cdot7,}

\auxidentity{U_{B_5}(p_0t^{-5})=p_0\bigl(-7^{11}t^6-8\cdot7^9t^5+3\cdot7^9t^4+30\cdot7^7t^3+81\cdot7^5t^2+30\cdot7^3t+3\cdot7\bigr)+p_1\bigl(7^{10}t^5+7^9t^4-3\cdot7^8t^3-27\cdot7^6t^2-9\cdot7^5t-12\cdot7^2\bigr)-7^{10}t^5-7^9t^4+9\cdot7^7t^3+93\cdot7^5t^2+249\cdot7^3t+836\cdot7,}

\auxidentity{U_{B_5}(p_0t^{-6})=p_0\bigl(-7^{13}t^7-8\cdot7^{11}t^6+2\cdot7^9t^4+20\cdot7^7t^3+54\cdot7^5t^2+20\cdot7^3t+2\cdot7\bigr)+p_1\bigl(7^{12}t^6+7^{11}t^5-2\cdot7^8t^3-18\cdot7^6t^2-6\cdot7^5t-8\cdot7^2\bigr)-7^{12}t^6-7^{11}t^5+6\cdot7^7t^3+62\cdot7^5t^2+166\cdot7^3t-748\cdot7^2.}

\par\bigskip
\noindent\textbf{Group V}\par\nobreak\medskip

\auxidentity{U_{B_5}(p_1t^{-1})=p_0\bigl(8\cdot7^8t^4+80\cdot7^6t^3+216\cdot7^4t^2+552\cdot7t\bigr)+p_1\bigl(7^{11}t^5+19\cdot7^9t^4+18\cdot7^8t^3+327\cdot7^5t^2+34\cdot7^4t+29\bigr)+7^{10}t^5+11\cdot7^8t^4+38\cdot7^6t^3+31\cdot7^4t^2+41\cdot7t-1,}

\auxidentity{U_{B_5}(p_1t^{-2})=p_0\bigl(-8\cdot7^3t^2-8\cdot7^2t\bigr)+p_1\bigl(8\cdot7^2t+7^2\bigr)-7^3t^2-8\cdot7^2t+11,}

\auxidentity{U_{B_5}(p_1t^{-3})=p_0\bigl(-8\cdot7^5t^3-8\cdot7^4t^2\bigr)+p_1\bigl(8\cdot7^4t^2+7^4t\bigr)-7^5t^3-8\cdot7^4t^2-7^4t-96,}

\auxidentity{U_{B_5}(p_1t^{-4})=p_0\bigl(48\cdot7^7t^4+72\cdot7^6t^3+216\cdot7^4t^2+80\cdot7^2t+8\bigr)+p_1\bigl(-48\cdot7^6t^3-65\cdot7^5t^2-24\cdot7^4t-32\cdot7\bigr)-7^7t^4+16\cdot7^6t^3+199\cdot7^4t^2+664\cdot7^2t+740,}

\auxidentity{U_{B_5}(p_1t^{-5})=p_0\bigl(-8\cdot7^9t^5-64\cdot7^8t^4-80\cdot7^7t^3-216\cdot7^5t^2-80\cdot7^3t-8\cdot7\bigr)+p_1\bigl(8\cdot7^8t^4+9\cdot7^8t^3+72\cdot7^6t^2+24\cdot7^5t+32\cdot7^2\bigr)-7^9t^5-8\cdot7^8t^4-31\cdot7^7t^3-248\cdot7^5t^2-664\cdot7^3t-740\cdot7,}

\auxidentity{U_{B_5}(p_1t^{-6})=p_0\bigl(-8\cdot7^{11}t^6-8\cdot7^{10}t^5+40\cdot7^9t^4+400\cdot7^7t^3+1080\cdot7^5t^2+400\cdot7^3t+40\cdot7\bigr)+p_1\bigl(8\cdot7^{10}t^5+7^{10}t^4-40\cdot7^8t^3-360\cdot7^6t^2-120\cdot7^5t-160\cdot7^2\bigr)-7^{11}t^6-8\cdot7^{10}t^5-7^{10}t^4+120\cdot7^7t^3+1240\cdot7^5t^2+3320\cdot7^3t+4720\cdot7,}

\auxidentity{U_{B_5}(p_1t^{-7})=p_0\bigl(-8\cdot7^{13}t^7-8\cdot7^{12}t^6-136\cdot7^9t^4-1360\cdot7^7t^3-3672\cdot7^5t^2-1360\cdot7^3t-136\cdot7\bigr)+p_1\bigl(8\cdot7^{12}t^6+6\cdot7^{11}t^5-19\cdot7^9t^4+818\cdot7^7t^3+1167\cdot7^6t^2+2798\cdot7^4t+11\cdot7^4+t^{-1}\bigr)+6\cdot7^{13}t^7+20\cdot7^{12}t^6+272\cdot7^{10}t^5+1893\cdot7^8t^4+435\cdot7^7t^3-2953\cdot7^5t^2-10620\cdot7^3t-27029\cdot7.}

\par\bigskip
\noindent\textbf{Group VI}\par\nobreak\medskip

\auxidentity{U_{B_5}(1)=1,}

\auxidentity{U_{B_5}(t^{-1})=-7t-4,}

\auxidentity{U_{B_5}(t^{-2})=-7^3t^2+20,}

\auxidentity{U_{B_5}(t^{-3})=-7^5t^3-88,}

\auxidentity{U_{B_5}(t^{-4})=-7^7t^4+260,}

\auxidentity{U_{B_5}(t^{-5})=-7^9t^5+68\cdot7,}

\auxidentity{U_{B_5}(t^{-6})=-7^{11}t^6-2392\cdot7.}

\setcounter{section}{2}
\section[Fundamental relations for a7(n)]{Fundamental relations for \texorpdfstring{$a_7(n)$}{a7(n)}}\label{app:a7}

The six groups below give the identities for $U_{A_7}$ and $U_{B_7}$ applied to $p_0t^j$, $p_1t^j$, and $t^j$.

\par\bigskip
\noindent\textbf{Group I}\par\nobreak\medskip

\auxidentity{U_{A_7}(p_0)=p_0\bigl(7^8t^5+34\cdot7^6t^4+405\cdot7^4t^3+2150\cdot7^2t^2+688\cdot7t+8\cdot7\bigr)+p_1\bigl(-7^7t^4-33\cdot7^5t^3-381\cdot7^3t^2-40\cdot7^3t-87\cdot7-7t^{-1}\bigr)+9\cdot7^6t^4+255\cdot7^4t^3+393\cdot7^3t^2+1968\cdot7t+87\cdot7+7t^{-1},}

\auxidentity{U_{A_7}(p_0t^{-1})=p_0\bigl(7^3t^2+24\cdot7t+16\bigr)+p_1\bigl(-7^2t-23-2t^{-1}\bigr)+9\cdot7t+23+2t^{-1},}

\auxidentity{U_{A_7}(p_0t^{-2})=p_0\bigl(-7^5t^3-44\cdot7^3t^2-464\cdot7t-208\bigr)+p_1\bigl(7^4t^2+43\cdot7^2t+432+26t^{-1}\bigr)-11\cdot7^3t^2-47\cdot7^2t-432-26t^{-1},}

\auxidentity{U_{A_7}(p_0t^{-3})=p_0\bigl(-7^7t^4-32\cdot7^5t^3-54\cdot7^3t^2+2096\cdot7t+192\cdot7\bigr)+p_1\bigl(7^6t^3+31\cdot7^4t^2+34\cdot7^2t-286\cdot7-24\cdot7t^{-1}\bigr)-11\cdot7^5t^3-209\cdot7^3t^2-8\cdot7^2t+286\cdot7+24\cdot7t^{-1},}

\auxidentity{U_{A_7}(p_0t^{-4})=p_0\bigl(26\cdot7^7t^4+741\cdot7^5t^3+5034\cdot7^3t^2+353\cdot7^2t-6432\bigr)+p_1\bigl(-26\cdot7^6t^3-717\cdot7^4t^2-652\cdot7^3t-194\cdot7+804t^{-1}\bigr)+2\cdot7^7t^4+262\cdot7^5t^3+107\cdot7^5t^2+4440\cdot7^2t+194\cdot7-804t^{-1},}

\auxidentity{U_{A_7}(p_0t^{-5})=p_0\bigl(7^{11}t^6+26\cdot7^9t^5-94\cdot7^7t^4-1170\cdot7^6t^3-10022\cdot7^4t^2-22922\cdot7^2t+16152+8t^{-1}\bigr)+p_1\bigl(-7^{10}t^5-25\cdot7^8t^4+16\cdot7^7t^3+163\cdot7^6t^2+9220\cdot7^3t+19769\cdot7-2020t^{-1}-t^{-2}\bigr)+7^{10}t^5+139\cdot7^7t^4-36\cdot7^7t^3-8539\cdot7^4t^2-64240\cdot7^2t-138375+2020t^{-1}+t^{-2},}

\auxidentity{U_{A_7}(p_0t^{-6})=p_0\bigl(7^{13}t^7+32\cdot7^{11}t^6+445\cdot7^9t^5+7342\cdot7^7t^4+16095\cdot7^6t^3+115958\cdot7^4t^2+315003\cdot7^2t+109656-24\cdot7t^{-1}\bigr)+p_1\bigl(-7^{12}t^6-31\cdot7^{10}t^5-425\cdot7^8t^4-1019\cdot7^7t^3-15420\cdot7^5t^2-106814\cdot7^3t-277603\cdot7-13686t^{-1}+3\cdot7t^{-2}\bigr)+11\cdot7^{11}t^6+297\cdot7^9t^5+571\cdot7^8t^4+8781\cdot7^6t^3+114308\cdot7^4t^2+749998\cdot7^2t+277579\cdot7+13686t^{-1}-3\cdot7t^{-2}.}

\par\bigskip
\noindent\textbf{Group II}\par\nobreak\medskip

\auxidentity{U_{A_7}(p_1)=p_0\bigl(-6\cdot7^{13}t^8-212\cdot7^{11}t^7-3198\cdot7^9t^6-3788\cdot7^8t^5-18486\cdot7^6t^4-52512\cdot7^4t^3-79404\cdot7^2t^2-7152\cdot7t-16\cdot7\bigr)+p_1\bigl(6\cdot7^{12}t^7+207\cdot7^{10}t^6+3053\cdot7^8t^5+3540\cdot7^7t^4+16935\cdot7^5t^3+47245\cdot7^3t^2+70359\cdot7t+128\cdot7^2+2\cdot7t^{-1}\bigr)-7^{13}t^8-10\cdot7^{12}t^7-1773\cdot7^9t^6-3333\cdot7^8t^5-25701\cdot7^6t^4-120081\cdot7^4t^3-331735\cdot7^2t^2-70375\cdot7t-128\cdot7^2-2\cdot7t^{-1},}

\auxidentity{U_{A_7}(p_1t^{-1})=p_0\bigl(-8\cdot7t-8\bigr)+p_1\bigl(8+t^{-1}\bigr)-7t-8-t^{-1},}

\auxidentity{U_{A_7}(p_1t^{-2})=p_0\bigl(8\cdot7^3t^2+120\cdot7t+64\bigr)+p_1\bigl(-8\cdot7^2t-113-8t^{-1}\bigr)+7^3t^2+64\cdot7t+113+8t^{-1},}

\auxidentity{U_{A_7}(p_1t^{-3})=p_0\bigl(4\cdot7^5t^3-48\cdot7^3t^2-160\cdot7^2t-624\bigr)+p_1\bigl(-4\cdot7^4t^2+51\cdot7^2t+1058+78t^{-1}\bigr)+7^5t^3+20\cdot7^3t^2-439\cdot7t-1058-78t^{-1},}

\auxidentity{U_{A_7}(p_1t^{-4})=p_0\bigl(-8\cdot7^7t^4-248\cdot7^5t^3-1528\cdot7^3t^2+288\cdot7^2t+544\cdot7\bigr)+p_1\bigl(8\cdot7^6t^3+241\cdot7^4t^2+1368\cdot7^2t-320\cdot7-68\cdot7t^{-1}\bigr)-7^7t^4-88\cdot7^5t^3-249\cdot7^4t^2-1296\cdot7^2t+320\cdot7+68\cdot7t^{-1},}

\auxidentity{U_{A_7}(p_1t^{-5})=p_0\bigl(6\cdot7^9t^5+292\cdot7^7t^4+762\cdot7^6t^3+5116\cdot7^4t^2+9322\cdot7^2t-14712-8t^{-1}\bigr)+p_1\bigl(-6\cdot7^8t^4-41\cdot7^7t^3-731\cdot7^5t^2-4680\cdot7^3t-7895\cdot7+1840t^{-1}+t^{-2}\bigr)+7^9t^5+78\cdot7^7t^4+367\cdot7^6t^3+5373\cdot7^4t^2+32488\cdot7^2t+55257-1840t^{-1}-t^{-2},}

\auxidentity{U_{A_7}(p_1t^{-6})=p_0\bigl(-4\cdot7^{11}t^6-232\cdot7^9t^5-5924\cdot7^7t^4-11280\cdot7^6t^3-70204\cdot7^4t^2-166408\cdot7^2t-12528+160t^{-1}\bigr)+p_1\bigl(4\cdot7^{10}t^5+229\cdot7^8t^4+822\cdot7^7t^3+1532\cdot7^6t^2+64370\cdot7^3t+145850\cdot7+1546t^{-1}-20t^{-2}\bigr)-7^{11}t^6-68\cdot7^9t^5-2323\cdot7^7t^4-6810\cdot7^6t^3-78618\cdot7^4t^2-450958\cdot7^2t-1020790-1546t^{-1}+20t^{-2}.}

\par\bigskip
\noindent\textbf{Group III}\par\nobreak\medskip

\auxidentity{U_{A_7}(1)=p_0\bigl(6\cdot7^{14}t^8+212\cdot7^{12}t^7+3198\cdot7^{10}t^6+3788\cdot7^9t^5+18486\cdot7^7t^4+52512\cdot7^5t^3+79404\cdot7^3t^2+50072\cdot7t+16\cdot7^2\bigr)+p_1\bigl(-6\cdot7^{13}t^7-207\cdot7^{11}t^6-3053\cdot7^9t^5-3540\cdot7^8t^4-16935\cdot7^6t^3-330714\cdot7^3t^2-70359\cdot7^2t-6273\cdot7-2\cdot7^2t^{-1}\bigr)+7^{14}t^8+10\cdot7^{13}t^7+1773\cdot7^{10}t^6+3333\cdot7^9t^5+25701\cdot7^7t^4+120081\cdot7^5t^3+2322146\cdot7^2t^2+70375\cdot7^2t+6273\cdot7+2\cdot7^2t^{-1},}

\auxidentity{U_{A_7}(t^{-1})=p_0\bigl(48\cdot7t+48\bigr)+p_1\bigl(-48-6t^{-1}\bigr)+7^2t+48+6t^{-1},}

\auxidentity{U_{A_7}(t^{-2})=p_0\bigl(-80\cdot7t-64\bigr)+p_1\bigl(78+8t^{-1}\bigr)+7^3t^2-8\cdot7t-78-8t^{-1},}

\auxidentity{U_{A_7}(t^{-3})=p_0\bigl(12\cdot7^5t^3+312\cdot7^3t^2+400\cdot7^2t+160\cdot7\bigr)+p_1\bigl(-12\cdot7^4t^2-302\cdot7^2t-370\cdot7-20\cdot7t^{-1}\bigr)+7^5t^3+108\cdot7^3t^2+2264\cdot7t+370\cdot7+20\cdot7t^{-1},}

\auxidentity{U_{A_7}(t^{-4})=p_0\bigl(-216\cdot7^4t^2-2776\cdot7^2t-1240\cdot7+8t^{-1}\bigr)+p_1\bigl(7^4t^2+216\cdot7^3t+2585\cdot7+1084t^{-1}-t^{-2}\bigr)-7^7t^4-167\cdot7^3t^2-240\cdot7^3t-18087-1084t^{-1}+t^{-2},}

\auxidentity{U_{A_7}(t^{-5})=p_0\bigl(6\cdot7^9t^5+76\cdot7^7t^4-18\cdot7^7t^3-580\cdot7^4t^2+10898\cdot7^2t+51864-104t^{-1}\bigr)+p_1\bigl(-6\cdot7^8t^4-10\cdot7^7t^3+129\cdot7^5t^2+446\cdot7^3t-10475\cdot7-6470t^{-1}+13t^{-2}\bigr)+7^9t^5+62\cdot7^7t^4+54\cdot7^6t^3-871\cdot7^4t^2-302\cdot7^3t+73221+6470t^{-1}-13t^{-2},}

\auxidentity{U_{A_7}(t^{-6})=p_0\bigl(-12\cdot7^{11}t^6-440\cdot7^9t^5-100\cdot7^9t^4-104\cdot7^7t^3+17412\cdot7^4t^2+11112\cdot7^2t-33448\cdot7+904t^{-1}\bigr)+p_1\bigl(12\cdot7^{10}t^5+430\cdot7^8t^4+654\cdot7^7t^3+433\cdot7^5t^2-2290\cdot7^4t-5427\cdot7+29154t^{-1}-113t^{-2}\bigr)-7^{11}t^6-124\cdot7^9t^5-3410\cdot7^7t^4-4386\cdot7^6t^3-1965\cdot7^4t^2+107666\cdot7^2t+38893-29154t^{-1}+113t^{-2}.}

\par\bigskip
\noindent\textbf{Group IV}\par\nobreak\medskip

\auxidentity{U_{B_7}(qp_0)=p_0\bigl(8\cdot7^7t^4+80\cdot7^5t^3+216\cdot7^3t^2+79\cdot7t\bigr)+p_1\bigl(7^{10}t^5+19\cdot7^8t^4+18\cdot7^7t^3+327\cdot7^4t^2+34\cdot7^3t+4\bigr)+7^9t^5+11\cdot7^7t^4+38\cdot7^5t^3+31\cdot7^3t^2+6\cdot7t,}

\auxidentity{U_{B_7}(qp_0t^{-1})=p_0\bigl(-7^3t^2-8\cdot7t\bigr)+p_1\bigl(7^2t+7\bigr)-7^2t+1,}

\auxidentity{U_{B_7}(qp_0t^{-2})=p_0\bigl(-7^5t^3-8\cdot7^3t^2\bigr)+p_1\bigl(7^4t^2+7^3t\bigr)-7^4t^2-7^3t-12,}

\auxidentity{U_{B_7}(qp_0t^{-3})=p_0\bigl(6\cdot7^7t^4+62\cdot7^5t^3+27\cdot7^4t^2+10\cdot7^2t+1\bigr)+p_1\bigl(-6\cdot7^6t^3-8\cdot7^5t^2-3\cdot7^4t-4\cdot7\bigr)+2\cdot7^6t^3+24\cdot7^4t^2+83\cdot7^2t+108,}

\auxidentity{U_{B_7}(qp_0t^{-4})=p_0\bigl(-7^9t^5-50\cdot7^7t^4-60\cdot7^6t^3-162\cdot7^4t^2-60\cdot7^2t-6\bigr)+p_1\bigl(7^8t^4+7^8t^3+54\cdot7^5t^2+18\cdot7^4t+24\cdot7\bigr)-7^8t^4-25\cdot7^6t^3-186\cdot7^4t^2-498\cdot7^2t-120\cdot7,}

\auxidentity{U_{B_7}(qp_0t^{-5})=p_0\bigl(-7^{11}t^6-8\cdot7^9t^5+3\cdot7^9t^4+30\cdot7^7t^3+81\cdot7^5t^2+30\cdot7^3t+3\cdot7\bigr)+p_1\bigl(7^{10}t^5+7^9t^4-3\cdot7^8t^3-27\cdot7^6t^2-9\cdot7^5t-12\cdot7^2\bigr)-7^{10}t^5-7^9t^4+9\cdot7^7t^3+93\cdot7^5t^2+249\cdot7^3t+836\cdot7,}

\auxidentity{U_{B_7}(qp_0t^{-6})=p_0\bigl(-7^{13}t^7-8\cdot7^{11}t^6+2\cdot7^9t^4+20\cdot7^7t^3+54\cdot7^5t^2+20\cdot7^3t+2\cdot7\bigr)+p_1\bigl(7^{12}t^6+7^{11}t^5-2\cdot7^8t^3-18\cdot7^6t^2-6\cdot7^5t-8\cdot7^2\bigr)-7^{12}t^6-7^{11}t^5+6\cdot7^7t^3+62\cdot7^5t^2+166\cdot7^3t-748\cdot7^2.}

\par\bigskip
\noindent\textbf{Group V}\par\nobreak\medskip

\auxidentity{U_{B_7}(qp_1t^{-1})=p_0\bigl(8\cdot7^8t^4+80\cdot7^6t^3+216\cdot7^4t^2+552\cdot7t\bigr)+p_1\bigl(7^{11}t^5+19\cdot7^9t^4+18\cdot7^8t^3+327\cdot7^5t^2+34\cdot7^4t+29\bigr)+7^{10}t^5+11\cdot7^8t^4+38\cdot7^6t^3+31\cdot7^4t^2+41\cdot7t-1,}

\auxidentity{U_{B_7}(qp_1t^{-2})=p_0\bigl(-8\cdot7^3t^2-8\cdot7^2t\bigr)+p_1\bigl(8\cdot7^2t+7^2\bigr)-7^3t^2-8\cdot7^2t+11,}

\auxidentity{U_{B_7}(qp_1t^{-3})=p_0\bigl(-8\cdot7^5t^3-8\cdot7^4t^2\bigr)+p_1\bigl(8\cdot7^4t^2+7^4t\bigr)-7^5t^3-8\cdot7^4t^2-7^4t-96,}

\auxidentity{U_{B_7}(qp_1t^{-4})=p_0\bigl(48\cdot7^7t^4+72\cdot7^6t^3+216\cdot7^4t^2+80\cdot7^2t+8\bigr)+p_1\bigl(-48\cdot7^6t^3-65\cdot7^5t^2-24\cdot7^4t-32\cdot7\bigr)-7^7t^4+16\cdot7^6t^3+199\cdot7^4t^2+664\cdot7^2t+740,}

\auxidentity{U_{B_7}(qp_1t^{-5})=p_0\bigl(-8\cdot7^9t^5-64\cdot7^8t^4-80\cdot7^7t^3-216\cdot7^5t^2-80\cdot7^3t-8\cdot7\bigr)+p_1\bigl(8\cdot7^8t^4+9\cdot7^8t^3+72\cdot7^6t^2+24\cdot7^5t+32\cdot7^2\bigr)-7^9t^5-8\cdot7^8t^4-31\cdot7^7t^3-248\cdot7^5t^2-664\cdot7^3t-740\cdot7,}

\auxidentity{U_{B_7}(qp_1t^{-6})=p_0\bigl(-8\cdot7^{11}t^6-8\cdot7^{10}t^5+40\cdot7^9t^4+400\cdot7^7t^3+1080\cdot7^5t^2+400\cdot7^3t+40\cdot7\bigr)+p_1\bigl(8\cdot7^{10}t^5+7^{10}t^4-40\cdot7^8t^3-360\cdot7^6t^2-120\cdot7^5t-160\cdot7^2\bigr)-7^{11}t^6-8\cdot7^{10}t^5-7^{10}t^4+120\cdot7^7t^3+1240\cdot7^5t^2+3320\cdot7^3t+4720\cdot7,}

\auxidentity{U_{B_7}(qp_1t^{-7})=p_0\bigl(-8\cdot7^{13}t^7-8\cdot7^{12}t^6-136\cdot7^9t^4-1360\cdot7^7t^3-3672\cdot7^5t^2-1360\cdot7^3t-136\cdot7\bigr)+p_1\bigl(8\cdot7^{12}t^6+6\cdot7^{11}t^5-19\cdot7^9t^4+818\cdot7^7t^3+1167\cdot7^6t^2+2798\cdot7^4t+11\cdot7^4+t^{-1}\bigr)+6\cdot7^{13}t^7+20\cdot7^{12}t^6+272\cdot7^{10}t^5+1893\cdot7^8t^4+435\cdot7^7t^3-2953\cdot7^5t^2-10620\cdot7^3t-27029\cdot7.}

\par\bigskip
\noindent\textbf{Group VI}\par\nobreak\medskip

\auxidentity{U_{B_7}(q)=1,}

\auxidentity{U_{B_7}(qt^{-1})=-7t-4,}

\auxidentity{U_{B_7}(qt^{-2})=-7^3t^2+20,}

\auxidentity{U_{B_7}(qt^{-3})=-7^5t^3-88,}

\auxidentity{U_{B_7}(qt^{-4})=-7^7t^4+260,}

\auxidentity{U_{B_7}(qt^{-5})=-7^9t^5+68\cdot7,}

\auxidentity{U_{B_7}(qt^{-6})=-7^{11}t^6-2392\cdot7.}

\subsection*{Acknowledgments} 

The authors thank Dandan Chen for some helpful comments related to her paper \cite{Chen25}.

\subsection*{Declarations} The authors make the following declarations.
\begin{enumerate}
\item The first two authors are supported by a Start-Up Grant from Ahmedabad University (Reference No.~URBSASI24A5). The third author is supported by a Junior Research Fellowship from the ANRF via a Core Research Grant (File
No.~CRG/2023/002698).
 \item The authors acknowledge the use of ChatGPT 6 Astra in revising the notation and improving the language of the manuscript. The authors take full responsibility for the mathematical content, accuracy, and integrity of the paper.
    \item The authors declare no conflicts of interest.
    \item The authors have no relevant financial or non-financial interests to disclose.
    \item There are no datasets associated with this study.
    \item All authors contributed to the proofs and to the final writing of the manuscript.
    \item All authors read and approved the final manuscript.
\end{enumerate}

\bibliographystyle{alpha}
\bibliography{ref}

@article{CCG24,
  author  = {Chen, Dandan and Chen, Rong and Garvan, Frank},
  title   = {Congruences modulo powers of 5 and 7 for the crank and rank
             parity functions and related mock theta-functions},
  journal = {The Ramanujan Journal},
  volume  = {70},
  number  = {2},
  year    = {2026},
  pages   = {Paper No. 25, 54 pp.},
  doi     = {10.1007/s11139-026-01391-9},
  url     = {https://doi.org/10.1007/s11139-026-01391-9}
}

@misc{Chen25,
  author        = {Chen, Dandan},
  title         = {Congruences modulo powers of 7 for the reciprocal crank
                   parity function},
  year          = {2025},
  eprint        = {2511.18905},
  archivePrefix = {arXiv},
  primaryClass  = {math.CO},
  note          = {arXiv preprint arXiv:2511.18905},
  doi           = {10.48550/arXiv.2511.18905},
  url           = {https://arxiv.org/abs/2511.18905}
}

@article{CL04,
  author  = {Corteel, Sylvie and Lovejoy, Jeremy},
  title   = {Overpartitions},
  journal = {Transactions of the American Mathematical Society},
  volume  = {356},
  number  = {4},
  year    = {2004},
  pages   = {1623--1635},
  doi     = {10.1090/S0002-9947-03-03328-2},
  url     = {https://doi.org/10.1090/S0002-9947-03-03328-2}
}

@article{HS25,
  author  = {Hirschhorn, Michael D. and Sellers, James A.},
  title   = {A family of congruences modulo 7 for partitions with
             monochromatic even parts and multi-colored odd parts},
  journal = {Australasian Journal of Combinatorics},
  volume  = {95},
  number  = {3},
  year    = {2026},
  pages   = {435--442},
  url     = {https://ajc.maths.uq.edu.au/?page=get_volumes&volume=95}
}

@article{PauleRadu2012,
  author  = {Paule, Peter and Radu, Cristian-Silviu},
  title   = {The {Andrews--Sellers} family of partition congruences},
  journal = {Advances in Mathematics},
  volume  = {230},
  number  = {3},
  year    = {2012},
  pages   = {819--838},
  doi     = {10.1016/j.aim.2012.02.026},
  url     = {https://doi.org/10.1016/j.aim.2012.02.026}
}

@book{Andrews1998,
  author    = {Andrews, George E.},
  title     = {The Theory of Partitions},
  series    = {Cambridge Mathematical Library},
  publisher = {Cambridge University Press},
  address   = {Cambridge},
  year      = {1998},
  isbn      = {978-0-521-63766-4},
  doi       = {10.1017/CBO9780511608650},
  url       = {https://doi.org/10.1017/CBO9780511608650}
}

@article{Dyson1944,
  author  = {Dyson, F. J.},
  title   = {Some guesses in the theory of partitions},
  journal = {Eureka},
  number  = {8},
  year    = {1944},
  pages   = {10--15}
}

@article{Ramanujan1919,
  author  = {Ramanujan, Srinivasa},
  title   = {Some properties of {$p(n)$}, the number of partitions of {$n$}},
  journal = {Proceedings of the Cambridge Philosophical Society},
  volume  = {19},
  year    = {1919},
  pages   = {207--210}
}

@article{AndrewsGarvan1988,
  author  = {Andrews, George E. and Garvan, Frank G.},
  title   = {{Dyson}'s crank of a partition},
  journal = {Bulletin of the American Mathematical Society (N.S.)},
  volume  = {18},
  number  = {2},
  year    = {1988},
  pages   = {167--171},
  doi     = {10.1090/S0273-0979-1988-15637-6},
  url     = {https://doi.org/10.1090/S0273-0979-1988-15637-6}
}

@incollection{Garvan1988,
  author    = {Garvan, F. G.},
  title     = {Combinatorial interpretations of {Ramanujan}'s partition
               congruences},
  booktitle = {Ramanujan Revisited},
  publisher = {Academic Press},
  address   = {Boston, MA},
  year      = {1988},
  pages     = {29--45},
  isbn      = {0-12-058560-X}
}

@article{AmdeberhanMerca2025,
  author  = {Amdeberhan, Tewodros and Merca, Mircea},
  title   = {From crank to congruences},
  journal = {Mediterranean Journal of Mathematics},
  volume  = {22},
  number  = {8},
  year    = {2025},
  pages   = {Paper No. 190, 22 pp.},
  doi     = {10.1007/s00009-025-02955-7},
  url     = {https://doi.org/10.1007/s00009-025-02955-7}
}

@article{Maier2009,
  author        = {Maier, Robert S.},
  title         = {On rationally parametrized modular equations},
  journal       = {Journal of the Ramanujan Mathematical Society},
  volume        = {24},
  number        = {1},
  year          = {2009},
  pages         = {1--73},
  eprint        = {math/0611041},
  archivePrefix = {arXiv},
  primaryClass  = {math.NT},
  url           = {https://jrms.ramanujanmathsociety.org/archieves/v24-1.html}
}

@article{AtkinLehner1970,
  author  = {Atkin, A. O. L. and Lehner, J.},
  title   = {{Hecke} operators on {$\Gamma_0(m)$}},
  journal = {Mathematische Annalen},
  volume  = {185},
  year    = {1970},
  pages   = {134--160},
  doi     = {10.1007/BF01359701},
  url     = {https://doi.org/10.1007/BF01359701}
}

@incollection{GarvanQseries,
  author    = {Garvan, Frank},
  title     = {A {$q$}-product tutorial for a {$q$}-series {Maple} package},
  editor    = {Foata, Dominique and Han, Guo-Niu},
  booktitle = {The Andrews Festschrift: Seventeen Papers on Classical
               Number Theory and Combinatorics},
  publisher = {Springer},
  address   = {Berlin, Heidelberg},
  year      = {2001},
  pages     = {111--138},
  doi       = {10.1007/978-3-642-56513-7_4},
  url       = {https://doi.org/10.1007/978-3-642-56513-7_4}
}

@article{JSF,
  author        = {Johnson, Anjelin Mariya and Sellers, James A.
                   and Fathima, S. N.},
  title         = {Additional congruences for generalized color
                   partitions of {Hirschhorn} and {Sellers}},
  journal       = {Arabian Journal of Mathematics},
  year          = {forthcoming},
  eprint        = {2510.25250},
  archivePrefix = {arXiv},
  primaryClass  = {math.NT},
  url           = {https://arxiv.org/abs/2510.25250}
}

@misc{TSF26,
  author        = {Thejitha, M. P. and Sellers, James A.
                   and Fathima, S. N.},
  title         = {Arithmetic properties of colored partitions
                   restricted by parity of the parts},
  year          = {2026},
  howpublished  = {arXiv preprint arXiv:2601.15680},
  eprint        = {2601.15680},
  archivePrefix = {arXiv},
  primaryClass  = {math.CO},
  doi           = {10.48550/arXiv.2601.15680},
  url           = {https://arxiv.org/abs/2601.15680}
}

@misc{JS26,
  author        = {Jameson, Marie and Sellers, James A.},
  title         = {A uniform proof approach for congruences modulo {3}
                   for partitions with {$k$}-colored odd parts},
  year          = {2026},
  howpublished  = {arXiv preprint arXiv:2609.19440},
  eprint        = {2609.19440},
  archivePrefix = {arXiv},
  primaryClass  = {math.NT},
  doi           = {10.48550/arXiv.2609.19440},
  url           = {https://arxiv.org/abs/2609.19440}
}

@misc{GarvanETA,
  author        = {Garvan, Frank},
  title         = {A tutorial for the {MAPLE} {ETA} package},
  year          = {2019},
  eprint        = {1907.09130},
  archivePrefix = {arXiv},
  primaryClass  = {math.CO},
  note          = {arXiv preprint arXiv:1907.09130},
  doi           = {10.48550/arXiv.1907.09130},
  url           = {https://arxiv.org/abs/1907.09130}
}

\end{document}